\documentclass{article}
\usepackage[noadjust]{cite}
\usepackage[title]{appendix}
\usepackage{etex}
\usepackage{xcolor}
\usepackage{amsthm}
\usepackage{amsfonts}
\usepackage{amssymb}
\usepackage{amsgen}
\usepackage{amsmath}
\usepackage{amsopn}
\usepackage{verbatim}
\usepackage{xypic}
\usepackage{pgf}
\usepackage{xspace}
\usepackage{multicol}
\usepackage{makeidx}
\usepackage{eepic}
\usepackage{upref}
\usepackage{pgf}
\usepackage{tikz}
\usepackage[normalem]{ulem}
\usepackage{shuffle,yfonts}
\DeclareFontFamily{U}{shuffle}{}
\DeclareFontShape{U}{shuffle}{m}{n}{ <-8>shuffle7 <8->shuffle10}{}

\allowdisplaybreaks

\newcommand{\FES}{\mathsf {FES}}
\newcommand{\FMZV}{\mathsf {FMZ}}

\newcommand{\ES}{\mathsf {ES}}
\newcommand{\MZV}{\mathsf {MZ}}

\newcommand{\sha}{\shuffle}
\newcommand{\fkeqv}{{\text{``$\equiv$''}}}
\newcommand{\fkeq}{{\text{``$=$''}}}

\newcommand{\Sy}{{\mathcal S}}

\newcommand{\gb}{\beta}

\newcommand\ta{{\texttt{a}}}
\newcommand\tb{{\texttt{b}}}
\newcommand\tc{{\texttt{c}}}
\newcommand{\tq}{{\texttt{q}}}

\newcommand\eps{{\varepsilon}}

\newcommand{\bfs}{{\boldsymbol{\sl{s}}}}

\newcommand\bfeps{{\boldsymbol \varepsilon}}

\newcommand{\calA}{\mathcal{A}}

\newcommand{\calP}{\mathcal{P}}
\newcommand{\eff}{\text{\rm{ef}}}

\allowdisplaybreaks

\newcommand{\N}{\mathbb{N}}
\newcommand{\Z}{\mathbb{Z}}
\newcommand{\Q}{\mathbb{Q}}

\newcommand{\ol}{\overline}

\def\ppmod#1{{\ (\rm{mod}\ 2)}}

\theoremstyle{plain}
\newtheorem{thm}{Theorem}[section]
\newtheorem{lem}[thm]{Lemma}
\newtheorem{ques}[thm]{Question}
\newtheorem{conj}[thm]{Conjecture}

\newtheorem{cor}[thm]{Corollary}
\newtheorem{prop}[thm]{Proposition}

\theoremstyle{definition}
\newtheorem{defn}{Definition}[section]
\newtheorem{conj-defn}{Definition-Conjecture}[section]
\newtheorem{rem}[thm]{Remark}
\newtheorem{eg}[thm]{Example}

\newtheorem{KeyHyp}[thm]{Key Hypothesis}

\begin{document}
\title{Effective Euler Sums}

\author{Jianqiang Zhao\footnote{Email: zhaoj@ihes.fr.}\\
\ \\
Department of Mathematics, The Bishop's School, La Jolla, CA 92037, USA}

\date{}

\maketitle

\medskip
\noindent
\textbf{Abstract.} Euler sums, also called alternating multiple zeta values (MZVs), have been shown to play important roles in many research areas in mathematics and theoretical physics. Motivated by a similar conjecture for MZVs by Kaneko and Zagier, the author proposed a conjectural isomorphism between the space of finite Euler sums and the space of classical ones modulo $\zeta(2)$-products. In this paper, we explicitly construct those finite elements (for which we call effective Euler sums) that correspond to the classical Euler sums under this conjecture in depth one and two. In the appendix, we offer another group of plausible candidates of effective Euler sums when depth is two and weight is even, by extending the heuristic argument of Kaneko and Zagier for the double zeta case. Kina recently obtained independently some the same results in the MZV setting via a completely different approach. Using one of his results we prove in the appendix that Kaneko and Zagier's heuristically defined effective double zeta values coincides with ours.



\section{Introduction, Motivation, and Terminology}
\subsection{Multiple zeta values.}
Multiple zeta values (MZVs) occupy a remarkable position at the intersection of number theory, combinatorics, algebra, and mathematical physics. For $d$-tuples of positive integers $\bfs=(s_1,\dots,s_d)\in\N^d$, with $s_1\geq 2$, the \emph{multiple zeta value}
\begin{equation}\label{equ:MZVdefn}
\zeta(\bfs)=\sum_{n_1>\cdots>n_d>0} \frac{1}{n_1^{s_1}\cdots n_d^{s_d}}
\end{equation}
is a seemingly simple nested sum whose algebraic structure is nonetheless extraordinarily rich. Kontsevitch observed that they were even equipped with
iterated integral expressions
\begin{equation}\label{equ:MZVintegral}
\zeta(\bfs)=\int_0^1 \ta^{s_1-1} \tb\cdots \ta^{s_d-1} \tb,
\end{equation}
where the 1-forms
\begin{equation*}
    \ta=\frac{dt}{t},\quad   \tb=\frac{dt}{1-t}.
\end{equation*}
The \emph{weight} $s_1+\cdots+s_d$ and the \emph{depth} $d$ provide a natural bigrading, while the two fundamental products, namely the stuffle product arising from \eqref{equ:MZVdefn} and the shuffle product arising from \eqref{equ:MZVintegral} via Chen's theory iterated integrals explained in \cite{KTChen1971,KTChen1977}, impose a large and highly nontrivial system of relations among these numbers. The study of these relations has led to a rich theory involving mixed Tate motives, periods, associators, modular forms, and Galois actions. See \cite{Brown2012,GanglKaZa2006,Glanois2015,Kemp2026,Racinet2002,Zagier1994}.

\subsection{Euler sums.}
A natural enlargement of the classical MZV algebra is obtained by allowing alternating signs in the summand. The resulting numbers are usually called Euler sums or alternating MZVs. More precisely, for $(s_1,\dots,s_d)\in\N^d$ and $(\eps_1,\dots,\eps_d)\in\{\pm 1\}^d$, we define the \emph{Euler sums} by
\begin{equation*}
\zeta(s_1,\dots,s_d;\eps_1,\dots,\eps_d):=\sum_{n_1>\cdots>n_d>0}\frac{\eps_1^{n_1}\cdots\eps_d^{n_d}}{n_1^{s_1}\cdots n_d^{s_d}}.
\end{equation*}
Already at depth one, these include the alternating zeta values
\begin{equation}\label{equ:ESdepth1}
\zeta(\bar{s})=\zeta(s;-1)=\sum_{n\geq1}\frac{(-1)^{n}}{n^s}=(2^{1-s}-1)\zeta(s).
\end{equation}
Here and throughout the paper, we suppress the signs in Euler sums by putting a bar on top of the corresponding $s_j$ if $\eps_j=-1$.

Note, however, at higher depth the situation becomes considerably subtler and an Euler sum usually cannot be expressed as a $\Q$-linear combination of MZVs, namely, Euler sums form a substantially larger space than the classical MZVs. Moreover, their relations already reflect the richer geometry of periods at general $N$-th roots of unity (the level $N=2$ for Euler sums) and a striking amount of structure survives for general level $N$: products, dualities, distribution relations, and parity phenomena continue to govern the algebra.

Euler sums have also acquired independent significance in mathematical physics. They occur naturally in the evaluation of Feynman integrals, in perturbative quantum field theory, and in the study of periods associated with graphs and iterated integrals (see, for e.g., \cite{BlumleinBrVe2010,BorweinBrBr1997,Broadhurst1996a,BroadhurstKr1995}). This has made explicit reduction formulas not merely an exercise in symbolic manipulation, but an important problem in the systematic understanding of the period algebra underlying these calculations.

\subsection{Finite multiple zeta values.}
A particularly intriguing development in the subject is the appearance of a parallel theory over finite fields (see \cite{KanekoZa2026,Hoffman2004,Zhao2008a,Zhao2011c}). Let $\calP$ be the set of all primes and put
\begin{equation*}
\calA:= \frac{\prod_{p\in\calP} \Z/p\Z}{\oplus_{p\in\calP} \Z/p\Z}.
\end{equation*}
Thus, $(a_p)_{p\in\calP}=(a_p)_{p\in\calP}$ in $\calA$ if $a_p=b_p$ for all sufficiently large primes $p$. We then define the \emph{finite multiple zeta values} by
\begin{equation*}
\zeta_\calA(\bfs):=\left( \sum_{p>n_1>\dots>n_d>0}  \frac{1}{n_1^{s_1}\cdots n_d^{s_d}} \pmod{p}\right)_{p: \calP}\in \calA.
\end{equation*}

The central philosophy behind the theory of finite MZVs was articulated most prominently by Kaneko and Zagier in \cite[Main Conjecture]{KanekoZa2026}. Let $\MZV$ denote the algebra of classical MZVs and $\FMZV$ the algebra of finite MZVs. Then they predict in \cite{KanekoZa2026} that there is an isomorphism
$$
\FMZV \cong \MZV/\zeta(2)\MZV.
$$
The significance of this conjecture is not simply that it predicts numerical identities. It suggests that the finite and classical theories are two manifestations of a common algebraic structure. In particular, a classical MZV should have a finite counterpart, and relations among classical MZVs should have finite analogues. Part of this correspondence (more precisely, the surjection from the left to the right assuming the map is well-defined) is guaranteed by Yasuda's result \cite{Yasuda2016}. However, the following seemingly naive-looking, yet fundamental, and in fact very difficult question does not have a definite answer yet:
\begin{ques}
Is there a nonzero element in $\FMZV$?
\end{ques}

\subsection{Finite Euler sum and the key conjecture.}
The same philosophy above naturally raises the question for Euler sums. We define the \emph{finite Euler sums} $\zeta_\calA(\bfs;\bfeps)$ as follows. For any $\bfs\in\N^d$ and $\bfeps\in\{\pm 1\}^d$,
\begin{equation}\label{equ:FMZdepth2}
\zeta_\calA(\bfs;\bfeps):=\left( \sum_{p>n_1>\dots>n_d>0}  \frac{\eps_1^{n_1}\cdots \eps_d^{n_d}}{n_1^{s_1}\cdots n_d^{s_d}} \pmod{p}\right)_{p: \calP}\in \calA.
\end{equation}
These infinite tuples then form a $\Q$-vector space denoted by $\FES$. For example, \cite[Thm. 3.1]{Zhao2008a} states that
\begin{equation}\label{equ:FMZVdepth2}
\zeta_\calA(s,t)=\left( (-1)^s \binom{s+t}{s} \frac{B_{p-s-t}}{s+t} \pmod{p} \right)_{p\in\calP,p>s+t}\in \calA.
\end{equation}
The author proposed the following key conjecture
(see \cite[Conjecture~8.6.9]{Zhao2016}) by extending the above conjecture of Kakeko and Zagier from MZVs to Euler sums.

\begin{conj}\label{conj:KanekoZagierAltVersion}
For any $w\in\N$, let $\FES_{w}$ (resp.\ $\ES_w$) be the $\Q$-vector space generated by
all finite Euler sums (resp.\ Euler sums) of weight $w$. Then, there is an isomorphism:
\begin{align}
f_\ES: \FES_{w} & \longrightarrow \frac{\ES_w}{\zeta(2)\ES_{w-2}} , \notag\\
 \zeta_{\calA}(\bfs;\bfeps) & \longmapsto \zeta_\ast^\Sy(\bfs;\bfeps)  \pmod{\zeta(2)}   \label{equ:defn_f_ES}
\end{align}
where
\begin{align*}
\zeta_\ast^\Sy(\bfs;\bfeps):=&\sum_{i=0}^d
\left(\prod_{j=1}^i (-1)^{s_j} \eps_j \right) \zeta_\ast(s_i,\dots,s_1;\eps_i,\dots,\eps_1) \zeta_\ast(s_{i+1},\dots,s_d;\eps_{i+1},\dots,\eps_d)
\end{align*}
where $\zeta_\ast(\bfs;\bfeps)$ is the stuffle-regularized value (which is a non-constant polynomial if $s_1=\eps_1=1$ and is equal to $\zeta(\bfs;\bfeps)$ otherwise).
\end{conj}
We remark that the symmetric Euler sum $\zeta_\ast^\Sy(\bfs;\bfeps)$ is always a constant real number. Moreover,
one can replace $\zeta_\ast^\Sy(\bfs;\bfeps)$ by the shuffle regularized value $\zeta_\sha^\Sy(\bfs;\bfeps)$\footnote{See \cite[\S 13.3.1]{Zhao2016} for the definition of the two different ways of regularizations.}
since the difference always lies in $\zeta(2)\ES_{w-2}$.

The formulation of this conjecture is deceptively simple. The alternating setting introduces new signs, new distribution relations, and a substantially larger space of periods. In particular, there is no canonical reason for a classical Euler sum to determine a unique finite Euler sum. Even after passing to the quotient by $\zeta(2)$-products, one must choose representatives compatible with the various algebraic structures. Thus, the conjectural isomorphism should not merely be regarded as an abstract vector-space correspondence. One would like to construct the correspondence explicitly.

This is the problem addressed in the present paper.

\subsection{Inverse images of Euler sums under $f_\ES$: effective Euler sums.}
An informative and effective way to approach the correspondence in Conjecture~\ref{conj:KanekoZagierAltVersion} is to reverse the usual question. We start with a classical Euler sum and seek an explicit finite element that should correspond to it. We call this element an ``\emph{effective Euler sum}''. The terminology is intended to emphasize that these elements are not merely existential representatives: they are explicit expressions designed to make the conjectural correspondence computable. However, \emph{we must assume Hypothesis~\ref{Hyp:Key} throughout this paper}, without which all the results in this paper become nonsense:
\begin{KeyHyp}\label{Hyp:Key}
The map $f_\ES$ in \eqref{equ:defn_f_ES} is well-defined.
\end{KeyHyp}

Moreover, we often say ``the'' effective Euler sums because such objects are unique by Conjecture~\ref{conj:KanekoZagierAltVersion}. There are several reasons why their construction is worthwhile. First, it provides concrete evidence for the conjectural isomorphism at the level of individual elements. Second, it makes it possible to compare the algebraic structures on the finite and classical sides directly. Third, effective representatives can be used to formulate and test identities that would otherwise remain inaccessible on the finite side. Finally, explicit formulas often reveal phenomena that are invisible from an abstract dimension argument. In this sense, the construction of effective Euler sums may be viewed as a finite analogue of choosing a particularly useful basis of periods.

The first batch of nontrivial cases occur in low depth. At depth one, the alternating situation is relatively transparent. The classical Euler sums $\zeta(\ol{s})$ are essentially ordinary zeta values by \eqref{equ:ESdepth1}, up to the familiar factor $2^{1-s}-1$, and the finite analogues can be described explicitly. Depth two is the first setting in which the genuinely nested nature of the sums becomes essential. It is also the first setting in which different algebraic structures such as stuffle multiplication, shuffle-type relations, duality, and parity interact in a substantial way. Consequently, depth two provides an ideal testing ground for the conjectural correspondence.

In this paper we construct effective Euler sums explicitly in depths one and two. The construction is designed to respect the natural weight and depth filtrations and, as far as possible, the algebraic relations satisfied by the corresponding classical Euler sums. Section~\ref{sec:depthOne} treats the depth one case. In depth two, the formulas exhibit a characteristic dependence on the parity of the weight and on the signs occurring in the Euler sum. This is not accidental. Parity has long played a central role in the reduction theory of Euler sums and MZVs: certain combinations of weight and depth force a reduction to lower-depth objects, while the complementary parity cases retain genuinely new information. The finite theory reflects this distinction, and it provides useful guidance for determining which classical expressions should serve as effective representatives.

Hence, when the depth is two but the weight is odd, we can reduce all such double Euler sums to depth one and we deal with this easy case in Section~\ref{sec:depthTwoOddWt}. However, when depth is two and the weight is even, we have to go up to depth three by first extending the parity result of Panzer from \cite{Panzer2017} to the non-admissible cases in Section~\ref{sec:Parity}. Then we will construct these (non-reducible) effective Euler sums using this information in Section~\ref{sec:depthTwoEvenWt}. In the appendix, motivated by a similar argument of Kaneko and Zagier for double zetas, we present an heuristic approach to defining effective double Euler sums of even weight and provide numerical evidence for its validity.

\medskip
\textbf{Acknowledgement.} The author thanks Prof. Huilin Zhu for his invitation to visit the Tianyuan Mathematical Center in Southeast China in July, 2026, supported by the funding number 12526102. The main results of this paper were presented at The 5th Workshop on Multiple Zeta Values and Related Fields, held in Changsha, China, in July 2026.

\section{Effective Finite Alternating Riemann Zeta Values}\label{sec:depthOne}
Recall that if $\bfs$ is non-admissible, then there are two ways to regularize $\zeta(\bfs)$ to get two possible polynomials (generally different): $\zeta_*(\bfs;T)$ by stuffle relations and $\zeta_\sha(\bfs;T)$ by shuffle relations. Therefore, the first guiding principle for finding the effect Euler sums is that there are two possible candidates $Z_{\calA,*}^\eff(\bfs)$ and $Z_{\calA,\sha}^\eff(\bfs)$ such that
\begin{equation*}
     f_\ES Z_{\calA,*}^\eff(\bfs)=\zeta_*(\bfs;0),\quad   f_\ES Z_{\calA,\sha}^\eff(\bfs)=\zeta_\sha(\bfs;0).
\end{equation*}

However, in depth one and two cases, $\zeta_*(\bfs;T)=\zeta_\sha(\bfs;T)$ except for $\bfs=(1,1)$.
But $\zeta_*(1,1;T)=(T^2-\zeta(2))/2\equiv T^2/2=\zeta_\sha(1,1;T)$ mod $\zeta(2)$. Thus, we only have one candidate $Z_\calA^\eff(\bfs)$ in these cases. In particular, as $\zeta_\sha(1;T)=\zeta_*(1;T)=T$ we can define $Z_\calA^\eff(1)=0$.

We now recall a well-known result.  Let $w>1$ be a positive odd number. Then we have (cf. \cite{KanekoZa2026})
\begin{equation}\label{equ:effRiemann}
    Z_\calA^\eff(w)=\big(Z_p^\eff(w)\big)_{p\in\calP,p>w}:=\left(  \frac{B_{p-w}}{w} \pmod{p} \right)_{p\in\calP,p>w}\in \calA.
\end{equation}
In fact, under Conjecture~\ref{conj:KanekoZagierAltVersion} we see easily from \eqref{equ:FMZVdepth2} that
\begin{align*}
    f_\ES( w Z_\calA^\eff(w))= f_\ES( \zeta_\calA(w-1,1) )= &\,\zeta_\ast^\Sy(w-1,1).
\end{align*}
If $w>2$ is odd then
\begin{align*}
    f_\ES( w Z_\calA^\eff(w))=&\, \zeta_\ast(w-1,1)+\zeta_\ast(w-1)\zeta_\ast(1)-\zeta_\ast(1,w-1) \\
    =&\,2\zeta(w-1,1) +\zeta(w)   \qquad \text{(by stuffle relation)}\\
    \equiv&\, n\zeta(w) \pmod{\zeta(2)}
\end{align*}
by the parity formula for double zeta values (see \cite[(75)]{BorweinBrBr1997} or \cite[Cor.~4.2]{XuZhao2023Oct}).
If $w$ is even then $Z_\calA^\eff(w)=0$ while
\begin{align*}
   \zeta_\ast^\Sy(w-1,1)=\zeta_\ast(w-1,1)-\zeta_\ast(w-1)\zeta_\ast(1)+\zeta_\ast(1,w-1)
   = -\zeta(w) \equiv 0 \pmod{\zeta(2)}.
\end{align*}
This shows that $Z_\calA^\eff(w)$ is indeed the correct \emph{effective} Riemann zeta value.

We now turn to the alternating version of \eqref{equ:effRiemann}.
\begin{prop}
For all $w\in\N$ the \emph{effective} finite alternating Riemann zeta value is
\begin{equation*}
    Z_\calA^\eff(\ol{w})=\big(Z_p^\eff(\ol{w})\big)_{p\in\calP,p>w+1}:=
\left\{
  \begin{array}{ll}
   \displaystyle   \left( (2^{1-w}-1) \frac{B_{p-w}}{w} \pmod{p} \right)_{p\in\calP,p>w+1} , & \hbox{if $w>1$;} \\
  \displaystyle   \emph{\tq}_2:=\left( \frac{2^{p-1}-1}{p} \pmod{p} \right)_{p\in\calP,p>2} , & \hbox{if $w=1$.}
  \end{array}
\right.
\end{equation*}
\end{prop}

\begin{proof}
By \cite[Thm. 8.2.7]{Zhao2016} for all $w\in\N$
\begin{align*}
    f_\ES\big( 2 Z_\calA^\eff(\ol{w})\big)=&\, f_\ES\big( \zeta_\calA(\ol{w})\big)= \zeta_\ast^\Sy(\ol{w})=2\zeta(\ol{w}).
\end{align*}
This verifies that $Z_\calA^\eff(\ol{w})$ is the correct effective alternating Riemann zeta value.
\end{proof}

In particular, $\tq_2$ corresponds to $-\zeta(\ol{1})=\log 2$ and, for all $w>1$
\begin{equation}\label{equ:ZbarByZ}
     Z_\calA^\eff(\ol{w})=(2^{1-w}-1)Z_\calA^\eff(w) \quad \text{v.s.}\quad \zeta(\ol{w})=(2^{1-w}-1)\zeta(w).
\end{equation}

As another evidence, by \cite[(8.52)]{Zhao2016} if $w>2$ is odd
\begin{align*}
    f_\ES( Z_\calA^\eff(\ol{w}))=&\, f_\ES( -\zeta_\calA(\ol{w-1},1) )= -\zeta_\ast^\Sy(\ol{w-1},1)\\
    =&\, -\zeta_\ast(\ol{w-1},1)+\zeta_\ast(\ol{w-1})\zeta_\ast(1)-\zeta_\ast(1,\ol{w-1}) = \zeta(\ol{w})
\end{align*}
by \cite[Thm. 4.2]{BCJXXZhao2020c}.

\section{Effective Finite Double Euler Sums of Odd Weight}\label{sec:depthTwoOddWt}
In this section, we will define all the \emph{effective} finite objects corresponding to double Euler sums of odd weight under $f_\ES$. We first recall a parity result on double Euler sums. Its motivic version can be found in \cite[Prop. 4.1]{XuZhao2023Oct} or \cite[Cor.~4.2.4]{Glanois2015}.

\begin{prop}\label{prop:DBESreduce}
Let $\eps_1,\eps_2=\pm1$, $a,b\in\N$ with $a+b=w$ odd. Then modulo $\zeta(2)$
\begin{align*}
\zeta_\ast(a,b;\eps_1,\eps_2)\equiv &\, -\frac12\zeta(w;\eps_1\eps_2)+\frac{(-1)^a}2\bigg[\binom{w-1}{a-1}\zeta(w;\eps_1)+\binom{w-1}{b-1}\zeta(w;\eps_2)\bigg].
\end{align*}
\end{prop}

We immediately obtain the following corollary by specialization.
\begin{cor}\label{cor:DBESreduce}
We have
\begin{align*}
\zeta_\ast(a,b)\equiv&\, -\frac12\zeta(w)+\frac{(-1)^a}{2}\binom{w}{a} \zeta(w),\\
\zeta_\ast(\ol{a},\ol{b})\equiv &\, -\frac12\zeta(w)+\frac{(-1)^a}2\binom{w}{a}\zeta(\ol{w}),\\
\zeta_\ast(\ol{a},b)\equiv &\, -\frac12\zeta(\ol{w})+\frac{(-1)^a}2\bigg[\binom{w-1}{a-1}\zeta(\ol{w})+\binom{w-1}{b-1}\zeta(w)\bigg],\\
\zeta_\ast(a,\ol{b})\equiv &\,-\frac12\zeta(\ol{w})+\frac{(-1)^a}2\bigg[\binom{w-1}{a-1}\zeta(w)+\binom{w-1}{b-1}\zeta(\ol{w})\bigg].
\end{align*}
\end{cor}

\begin{cor}\label{cor:ESWtSumFormula}
\label{conj:WtSumFormula}
For all positive integers $w\ge 3$, we have
\begin{align}
  \sum_{\substack{a\ge 2,b\ge 1\\ a+b=w}} 2^a \zeta(a,b)=&\, (w+1)\zeta(w), \notag\\
  \sum_{\substack{a\ge 2,b\ge 1\\ a+b=w}} 2^a \zeta(\ol{a},b)=&\, 2\zeta(\ol{w}) + (w-1)\zeta(w),  \notag\\
  \sum_{\substack{a\ge 2,b\ge 1\\ a+b=w}} 2^a \Big[\zeta(\ol{a},\ol{b})+\zeta(a,\ol{b})\Big]= &\, 2w\zeta(\ol{w})+2\zeta(w).
  \label{equ:WtSumES(a,barb)AND(a,barb)}.
\end{align}
Further, if $w$ is an odd integer and $w\ge 3$, then
\begin{align*}
  \sum_{\substack{a\ge 2,b\ge 1\\ a+b=w}} 2^a \zeta(\ol{a},\ol{b})\equiv &\, (w-1)\zeta(\ol{w})+(3-2^w)\zeta(w)  \pmod{\zeta(2)},\\
  \sum_{\substack{a\ge 2,b\ge 1\\ a+b=w}} 2^a \zeta(a,\ol{b})\equiv &\, (w+1)\zeta(\ol{w})+(2^w-1)\zeta(w)\pmod{\zeta(2)}.
\end{align*}
\end{cor}
\begin{proof}
The first three identities are given at the beginning of the proof of \cite[Theorem 4.4]{BCJXXZhao2020c}.
The last two follow easily from the identities
\begin{align}
  \sum_{\substack{a\ge 2,b\ge 1\\ a+b=w}}  2^a  =&\, 2^w-4,    \label{equ:CombId1}\\
  \sum_{\substack{a\ge 2,b\ge 1\\ a+b=w}}  (-2)^a \binom{w}{a}=&\, (1-2)^w-1+2w-(-2)^w=2^w+2w-2,  \label{equ:CombId2}\\
  \sum_{\substack{a\ge 2,b\ge 1\\ a+b=w}}  (-2)^a \binom{w-1}{a}=&\, (1-2)^{w-1}-1+2(w-1)=2w-2,  \label{equ:CombId3}\\
  \sum_{\substack{a\ge 2,b\ge 1\\ a+b=w}}  (-2)^a \binom{w-1}{a-1}=&\, -2\big((1-2)^{w-1}-1-(-2)^{w-1}\big)=2^w.  \label{equ:CombId4}
\end{align}
We remark that it is possible to obtain exact equations involving product of $\zeta(2)$ in these last two formulas because Prop.~\ref{prop:DBESreduce} can be lifted to an exact equation. We leave these to the interested reader.
\end{proof}

\begin{thm} \label{thm:effDBESoddWt}
Suppose $a\ge 2$ and $b\ge 1$ such that $w=a+b$ is odd. Then
\begin{align*}
Z_\calA^\eff(a,b)= &\, \frac12 \left[(-1)^a \binom{w}{a}-1\right]Z_\calA^\eff(w),\\ 
Z_\calA^\eff(\ol{a},\ol{b})= &\, \frac12 \left[(-1)^a (2^{1-w}-1)\binom{w}{a}-1\right]Z_\calA^\eff(w),\\ 
Z_\calA^\eff(\ol{a},b)= &\, \frac12\bigg\{1-2^{1-w}+(-1)^a \bigg[\binom{w-1}{a-1}(2^{1-w}-1) +\binom{w-1}{b-1}\bigg]\bigg\}Z_\calA^\eff(w),  \\ 
Z_\calA^\eff(a,\ol{b})= &\,\frac12\bigg\{1-2^{1-w} +(-1)^a \bigg[\binom{w-1}{a-1}+\binom{w-1}{b-1}(2^{1-w}-1)\bigg]\bigg\}Z_\calA^\eff(w). 
\end{align*}
\end{thm}

\begin{proof}
First, we observe that either $a$ or $b$ is even so that
\begin{equation*}
\zeta_\ast(a)\zeta_\ast(b)\equiv  \zeta_\ast(\ol{a})\zeta_\ast(\ol{b}) \equiv  \zeta_\ast(\ol{a})\zeta_\ast(b) \equiv  \zeta_\ast(a)\zeta_\ast(\ol{b}) \equiv  0 \pmod{\zeta(2)}.
\end{equation*}
By definition and Cor.~\ref{cor:DBESreduce}, modulo $\zeta(2)$
\begin{align*}
    \zeta_\ast^\Sy(a,b)=&\, \zeta_\ast(a,b)+(-1)^a \zeta_\ast(a)\zeta_\ast(b)-\zeta_\ast(b,a)\equiv2\zeta(a,b)+\zeta(w),\\
    \zeta_\ast^\Sy(\ol{a},\ol{b})=&\, \zeta_\ast(\ol{a},\ol{b})-(-1)^a \zeta_\ast(\ol{a})\zeta_\ast(\ol{b})-\zeta_\ast(\ol{b},\ol{a})\equiv2\zeta(\ol{a},\ol{b})+\zeta(w),\\
    \zeta_\ast^\Sy(\ol{a},b)=&\, \zeta_\ast(\ol{a},b)-(-1)^a \zeta_\ast(\ol{a})\zeta_\ast(b)+\zeta_\ast(b,\ol{a})\\
    \equiv &\, -\zeta(\ol{w})
    \equiv 2\zeta_\ast(\ol{a},b)-(-1)^a \bigg[\binom{w-1}{a-1}\zeta(\ol{w})+\binom{w-1}{b-1}\zeta(w)\bigg],,\\
    \zeta_\ast^\Sy(a,\ol{b})=&\, \zeta_\ast(a,\ol{b})+(-1)^a \zeta_\ast(a)\zeta_\ast(\ol{b})+\zeta_\ast(\ol{b},a)\\
   \equiv &\,  -\zeta(\ol{w})
    \equiv 2\zeta_\ast(a,\ol{b})- (-1)^a \bigg[\binom{w-1}{a-1}\zeta(w)+\binom{w-1}{b-1}\zeta(\ol{w})\bigg].
\end{align*}
We see that
\begin{align*}
Z_\calA^\eff(a,b)= &\, \frac12 \Big( \zeta_\calA(a,b)-Z_\calA^\eff(w)\Big),\\
Z_\calA^\eff(\ol{a},\ol{b})= &\, \frac12 \Big( \zeta_\calA(\ol{a},\ol{b})-Z_\calA^\eff(w) \Big)\\
Z_\calA^\eff(\ol{a},b)= &\,\frac12\bigg\{\zeta_\calA(\ol{a},b)+(-1)^a \bigg[\binom{w-1}{a-1}Z_\calA^\eff(\ol{w})+\binom{w-1}{b-1}Z_\calA^\eff(w)\bigg]\bigg\},\\
Z_\calA^\eff(a,\ol{b})= &\,\frac12\bigg\{\zeta_\calA(a,\ol{b})+(-1)^a \bigg[\binom{w-1}{a-1}Z_\calA^\eff(w)+\binom{w-1}{b-1}Z_\calA^\eff(\ol{w})\bigg]\bigg\}.
\end{align*}
The proposition follows from \cite[Thm. 8.6.4]{Zhao2016} at once.
\end{proof}

\begin{prop} \label{prop:StuffleRel} \emph{(Stuffle relations)}
For all admissible $(a,b)$ with odd weight $a+b$, we have the following relations:
\begin{align*}
  Z_\calA^\eff(a)Z_\calA^\eff(b)=&\, Z_\calA^\eff(a,b)+ Z_\calA^\eff(b,a)+ Z_\calA^\eff(a+b),\\
  Z_\calA^\eff(\ol{a})Z_\calA^\eff(b)=&\, Z_\calA^\eff(\ol{a},b)+ Z_\calA^\eff(b,\ol{a})+ Z_\calA^\eff(\ol{a+b}),\\
  Z_\calA^\eff(\ol{a})Z_\calA^\eff(\ol{b})=&\, Z_\calA^\eff(\ol{a},\ol{b})+ Z_\calA^\eff(\ol{b},\ol{a})+ Z_\calA^\eff(a+b).
\end{align*}
\end{prop}
\begin{proof}
These are straight-forward by \eqref{equ:ZbarByZ} and the fact that $(-1)^a+(-1)^b=0$.
\end{proof}

\begin{prop}\label{prop:SumFormula}   \emph{(Sum formulas})
For all positive odd integer $w\ge 3$, setting $v=w-1$ then we have
\begin{align*}
  \sum_{\substack{a\ge 2,b\ge 1\\ a+b=w}}  Z_\calA^\eff(a,b)=&\, Z_\calA^\eff(w),\\
  \sum_{\substack{a\ge 2,b\ge 1\\ a+b=w}}  Z_\calA^\eff(\ol{a},\ol{b})=&\,  Z_\calA^\eff(\ol{1},v)- Z_\calA^\eff(\ol{1},\ol{v})+Z_\calA^\eff(\ol{w}), \\
  \sum_{\substack{a\ge 2,b\ge 1\\ a+b=w}}  Z_\calA^\eff(\ol{a},b)=&\,  Z_\calA^\eff(\ol{v},\ol{1})+ Z_\calA^\eff(\ol{1},\ol{v})-Z_\calA^\eff(\ol{v},1)-Z_\calA^\eff(\ol{1},v)+Z_\calA^\eff(w), \\
  \sum_{\substack{a\ge 2,b\ge 1\\ a+b=w}}  Z_\calA^\eff(a,\ol{b})=&\, Z_\calA^\eff(\ol{v},1)-Z_\calA^\eff(\ol{v},\ol{1})+Z_\calA^\eff(\ol{w}).
\end{align*}
\end{prop}

\begin{rem}
One should compare these with the sum formulas of double Euler sums given in \cite[Theorem 4.2]{BCJXXZhao2020c} which have exactly the same form.
\end{rem}

\begin{proof}
By the binomial expansion
\begin{align*}
  \sum_{\substack{a\ge 2,b\ge 1\\ a+b=w}}  (-1)^a \binom{w}{a} =
  (1-1)^w-\binom{w}{0}+\binom{w}{1}-(-1)^w \binom{w}{w} =w
\end{align*}
we see that
\begin{align*}
  \sum_{\substack{a\ge 2,b\ge 1\\ a+b=w}}  Z_\calA^\eff(a,b)
  =&\, \frac12 Z_\calA^\eff(w)\sum_{\substack{a\ge 2,b\ge 1\\ a+b=w}}  \left[(-1)^a \binom{w}{a}-1\right]
  =Z_\calA^\eff(w).
\end{align*}
Similarly,
\begin{align*}
  \sum_{\substack{a\ge 2,b\ge 1\\ a+b=w}}  Z_\calA^\eff(\ol{a},\ol{b})
  =&\, \frac12 Z_\calA^\eff(w)\sum_{\substack{a\ge 2,b\ge 1\\ a+b=w}}  \left[(-1)^a (2^{1-w}-1)\binom{w}{a}-1 \right] \\
  =&\,  \frac12 Z_\calA^\eff(w)\left( w(2^{1-w}-1)-(w-2)  \right)
\end{align*}
which implies the sum formula for $Z_\calA^\eff(\ol{a},\ol{b})$ by Theorem~\ref{thm:effDBESoddWt}. The proof of the other two formulas are similar and are thus omitted.
\end{proof}

\begin{prop}\label{prop:WtSumFormula} \emph{(Weighted sum formulas)}
For all positive odd integer $w\ge 3$,
\begin{align}
  \sum_{\substack{a\ge 2,b\ge 1\\ a+b=w}} 2^a  Z_\calA^\eff(a,b)=&\, (w+1) Z_\calA^\eff(w), \label{equ:OddWtSum(a,b)}\\
  \sum_{\substack{a\ge 2,b\ge 1\\ a+b=w}} 2^a Z_\calA^\eff(\ol{a},\ol{b})=&\, (w-1)Z_\calA^\eff(\ol{w})+(3-2^w)Z_\calA^\eff(w),\label{equ:WtSum(bara,barb)}\\
  \sum_{\substack{a\ge 2,b\ge 1\\ a+b=w}} 2^a Z_\calA^\eff(\ol{a},b)=&\, 2Z_\calA^\eff(\ol{w}) + (w-1)Z_\calA^\eff(w),\notag \\
  \sum_{\substack{a\ge 2,b\ge 1\\ a+b=w}} 2^a Z_\calA^\eff(a,\ol{b})=&\, (w+1)Z_\calA^\eff(\ol{w})+(2^w-1)Z_\calA^\eff(w).\label{equ:WtSum(a,barb)}
\end{align}
Consequently,
\begin{align}\label{equ:WtSum(a,barb)AND(a,barb)}
\sum_{\substack{a\ge 2,b\ge 1\\ a+b=w}}  2^a \Big(Z_\calA^\eff(\ol{a},\ol{b})+ Z_\calA^\eff(a,\ol{b})\Big)
=&\,2 Z_\calA^\eff(w)+ 2w Z_\calA^\eff(\ol{w}).
\end{align}
\end{prop}

\begin{rem}
One should compare these with the weighted sum formulas of double Euler sums given in Cor.~\ref{cor:ESWtSumFormula}. Note that the corresponding formulas for \eqref{equ:WtSum(bara,barb)} and \eqref{equ:WtSum(a,barb)} hold modulo $\zeta(2)$ there but the sum of these two as given in \eqref{equ:WtSum(a,barb)AND(a,barb)} has an exact counterpart for double Euler sums as given by \eqref{equ:WtSumES(a,barb)AND(a,barb)}.
\end{rem}

\begin{proof}
By the binomial expansion identity \eqref{equ:CombId1} and  \eqref{equ:CombId2},
\begin{align*}
  \sum_{\substack{a\ge 2,b\ge 1\\ a+b=w}} 2^a Z_\calA^\eff(a,b)
  =&\, \frac12 Z_\calA^\eff(w)\sum_{\substack{a\ge 2,b\ge 1\\ a+b=w}}   2^a \left[(-1)^a \binom{w}{a}-1\right] \\
  =&\, \frac12 Z_\calA^\eff(w) \Big[2^w+2w-2-(2^w-4)\Big]
  =(w+1)Z_\calA^\eff(w).
\end{align*}
Similarly,
\begin{align*}
  \sum_{\substack{a\ge 2,b\ge 1\\ a+b=w}} 2^a Z_\calA^\eff(\ol{a},\ol{b})
  =&\, \frac12 Z_\calA^\eff(w)\sum_{\substack{a\ge 2,b\ge 1\\ a+b=w}}   2^a \Big[(-1)^a (2^{1-w}-1)\binom{w}{a}-1 \Big] \\
  =&\, \frac12 Z_\calA^\eff(w) \Big[(2^{1-w}-1)(2^w+2w-2)-(2^w-4)\Big]
\end{align*}
which implies the second formula in the proposition.

Next,  \eqref{equ:CombId1}, \eqref{equ:CombId3} and  \eqref{equ:CombId4} yield that
\begin{align*}
  \sum_{\substack{a\ge 2,b\ge 1\\ a+b=w}} 2^a Z_\calA^\eff(\ol{a},b)
  =&\, \frac12 Z_\calA^\eff(w)\sum_{\substack{a\ge 2,b\ge 1\\ a+b=w}}   2^a \bigg\{1-2^{1-w}+(-1)^a \bigg[\binom{w-1}{a-1}(2^{1-w}-1)+\binom{w-1}{b-1}\bigg]\bigg\} \\
  =&\, \frac12 Z_\calA^\eff(w) \Big[(1-2^{1-w})(2^w-4)+2^w(2^{1-w}-1) +2(w-1)\Big]\\
  =&\,  Z_\calA^\eff(w) \Big[2^{2-w}-3+w\Big]\\
  =&\, 2Z_\calA^\eff(\ol{w}) + (w-1)Z_\calA^\eff(w).
\end{align*}
In the same way, \eqref{equ:CombId1}, \eqref{equ:CombId3} and  \eqref{equ:CombId4} yield that
\begin{align*}
  \sum_{\substack{a\ge 2,b\ge 1\\ a+b=w}} 2^a Z_\calA^\eff(a,\ol{b})
  =&\, \frac12 Z_\calA^\eff(w)\sum_{\substack{a\ge 2,b\ge 1\\ a+b=w}}   2^a\bigg\{1-2^{1-w} +(-1)^a \bigg[\binom{w-1}{a-1}+\binom{w-1}{b-1}(2^{1-w}-1)\bigg]\bigg\}\\
  =&\, \frac12 Z_\calA^\eff(w) \Big[(1-2^{1-w})(2^w-4)+2^{w}+ 2 (w-1)(2^{1-w}-1) \Big]\\
  =&\,  Z_\calA^\eff(w) \Big[2^w-1+(2^{1-w}-1) (w+1)\Big]
\end{align*}
which implies \eqref{equ:WtSum(a,barb)}. Finally, \eqref{equ:WtSum(a,barb)AND(a,barb)} quickly follows from \eqref{equ:WtSum(bara,barb)} and \eqref{equ:WtSum(a,barb)}. This completes the proof of the proposition.
\end{proof}

\section{Parity Principle For Regularized Triple Euler Sums of Even Weight}\label{sec:Parity}
In this section, we generalize the parity result of Panzer for triple Euler sums to the divergent cases. In this section and throughout the rest of the paper, we always use $\eps_j$'s to represent  $\pm 1$ if we do not mention otherwise.

\begin{thm}\label{thm:PanzerDivES}
Let $a,b,c\in\N$ with $w=a+b+c$ even. Then modulo $\zeta(2)$ we have
\begin{align*}
&\,  2 \zeta_*(n_1,n_2,n_3;\eps_1,\eps_2,\eps_3) \\
\equiv &\,2\zeta_*(n_3;\eps_3)\zeta_*(n_1,n_2;\eps_1,\eps_2)-\zeta_*(n_1,n_2+n_3;\eps_1,\eps_2\eps_3)+\zeta_*(n_3,n_1+n_2;\eps_3,\eps_1\eps_2) \\
&\, +(-1)^{n_2} \sum_{\substack{\mu+\nu=w\\ \mu\ge n_1,\nu\ge n_3}} \binom{\mu-1}{n_1-1} \binom{\nu-1}{n_3-1} \zeta_*(\mu;\eps_1) \zeta_*(\nu;\eps_3) \\
&\, -(-1)^{n_1}\sum_{\substack{\mu+\nu=w\\ \mu\ge n_2,\nu\ge n_3}}  \binom{\mu-1}{n_2-1} \binom{\nu-1}{n_3-1} \zeta_*(\nu,\mu;\eps_3,\eps_2)    \\
&\, -(-1)^{n_3}\sum_{\substack{\mu+\nu=w\\ \mu\ge n_2,\nu\ge n_1}}  \binom{\mu-1}{n_2-1} \binom{\nu-1}{n_1-1} \zeta_*(\nu,\mu;\eps_1,\eps_2) .
\end{align*}
\end{thm}
\begin{proof}
When $a>1$, this is a level two case of \cite[(4.3)]{Panzer2017} proved by Panzer. Hence, we only need to deal with the divergent cases of the theorem.

When $(n_1,n_2,n_3)=(1,b,c)$ and $b>1$, by the stuffle relation
\begin{align*}
 \zeta_*(1,b,c;\eps_1,\eps_2,\eps_3)=\zeta_*(1;\eps_1)\zeta_*(b,c;\eps_2,\eps_3) -\zeta(b,1,c;\eps_2,\eps_1,\eps_3)-\zeta(b,c,1;\eps_2,\eps_3,\eps_1)\\
 -\zeta(b+1,c;\eps_1\eps_2,\eps_3)-\zeta(b,c+1;\eps_1,\eps_2\eps_3).
\end{align*}
The two triple zeta values on the right-hand side are both admissible and therefore can be dealt with using
Panzer's original formula. By a routine though a little tedious computation we can see that it suffices to
prove the following three expressions add up to
$2\zeta_*(1;\eps_1)\zeta_*(b,c;\eps_2,\eps_3)-2\zeta(b+1,c;\eps_1\eps_2,\eps_3)-2\zeta(b,c+1;\eps_2,\eps_1\eps_3)$ modulo $\zeta(2)$ (where $w=b+c+1$):
\begin{align*}
2\zeta_*(1,b,c;\eps_1,\eps_2,\eps_3)
\overset{?}{\equiv} &\, 2\zeta_*(c;\eps_3)\zeta_*(1,b;\eps_1,\eps_2)-\zeta_*(1,b+c;\eps_1,\eps_2\eps_3)+\zeta_*(c,b+1;\eps_3,\eps_1\eps_2) \\
&\, +(-1)^{b} \sum_{\mu+\nu=w}  \binom{\nu-1}{c-1} \zeta_*(\mu;\eps_1) \zeta_*(\nu;\eps_3)
    +(-1)^{b}\sum_{\mu+\nu=w}  \binom{\mu-1}{b-1}  \zeta_*(\nu,\mu;\eps_1,\eps_2)\\
&\,+c \zeta(c+1,b;\eps_3,\eps_2) +b\zeta_*(c,b+1;\eps_3,\eps_2)   ,\\
2\zeta(b,1,c;\eps_2,\eps_1,\eps_3)
\equiv &\, 2\zeta_*(c;\eps_3)\zeta_*(b,1;\eps_2,\eps_1)-\zeta_*(b,c+1;\eps_2,\eps_1\eps_3)+\zeta_*(c,b+1;\eps_3,\eps_2\eps_1) \\
&\, -c\zeta_*(b;\eps_2) \zeta_*(c+1;\eps_3)-b\zeta_*(b+1;\eps_2) \zeta_*(c;\eps_3) \\
&\, -(-1)^{b}\sum_{\mu+\nu=w}  \binom{\nu-1}{c-1} \zeta_*(\nu,\mu;\eps_3,\eps_1)
     +(-1)^{b}\sum_{\mu+\nu=w}  \binom{\nu-1}{b-1} \zeta_*(\nu,\mu;\eps_2,\eps_1),\\
2\zeta_*(b,c,1;\eps_2,\eps_3,\eps_1))
\equiv &\, 2\zeta_*(1;\eps_1)\zeta_*(b,c;\eps_2,\eps_3)-\zeta_*(b,c+1;\eps_2,\eps_1\eps_3)+\zeta_*(1,b+c;\eps_1,\eps_2\eps_3) \\
&\, -(-1)^{b} \sum_{\mu+\nu=w} \binom{\mu-1}{b-1} \zeta_*(\mu;\eps_2) \zeta_*(\nu;\eps_1)
      -(-1)^{b}\sum_{\mu+\nu=w}  \binom{\mu-1}{c-1} \zeta_*(\nu,\mu;\eps_1,\eps_3)    \\
&\, +b\zeta_*(b+1,c;\eps_2,\eps_3)+c\zeta_*(b,c+1;\eps_2,\eps_3) .
\end{align*}
Hence, the sum of the right-hand side of the above three equations divided by 2 is equal to
\begin{multline*}
 \zeta_*(c;\eps_3)\zeta_*(1,b;\eps_1,\eps_2)
 +\zeta_*(c;\eps_3)\zeta_*(b,1;\eps_2,\eps_1)-\zeta_*(b,c+1;\eps_2,\eps_1\eps_3) \\
   +\zeta_*(c,b+1;\eps_3,\eps_1\eps_2)+\zeta_*(1;\eps_1)\zeta_*(b,c;\eps_2,\eps_3) \\
\equiv   \zeta_*(1;\eps_1)\zeta_*(b,c;\eps_2,\eps_3)+ \zeta_*(c;\eps_3)\zeta_*(1;\eps_1)\zeta_*(b;\eps_2)- \zeta_*(c;\eps_3)\zeta_*(b+1;\eps_1\eps_2)-\zeta_*(b,c+1;\eps_2,\eps_1\eps_3).
\end{multline*}
But the term $\zeta_*(c;\eps_3)\zeta_*(1;\eps_1)\zeta_*(b;\eps_2)$ vanishes since either $b$ or $c$ is even. This proves the theorem when $(n_1,n_2,n_3)=(1,b,c)$ and $b>1$.

When $(n_1,n_2,n_3)=(1,1,a)$, we can use stuffle relations to get
\begin{align*}
 &\, \zeta_*(1,1,a;\eps_1,\eps_2,\eps_3)\equiv
\zeta(a;\eps_3)\zeta_*(1,1;\eps_1,\eps_2)-\zeta_*(1,a,1;\eps_1,\eps_3,\eps_2)\\
&\, \hskip4cm -\zeta(a+1,1;\eps_1\eps_3,\eps_2)-\zeta_*(1,a+1;\eps_1,\eps_2\eps_3)-\zeta_*(a,1,1;\eps_3,\eps_1,\eps_2)\\
\equiv &\, -\zeta_*(1;\eps_1)\zeta_*(a,1;\eps_3,\eps_2)
+\zeta_*(a,2;\eps_3,\eps_1\eps_2) +\zeta_*(a,1,1;\eps_3,\eps_2,\eps_1)
-\zeta_*(1,a+1;\eps_1,\eps_2\eps_3).
\end{align*}
Using the expression of $\zeta(a,1,1;\eps_3,\eps_2,\eps_1)$ proved already
\begin{align*}
 &\,  2 \zeta(a,1,1;\eps_3,\eps_2,\eps_1)
\equiv 2\zeta_*(1;\eps_1)\zeta_*(a,1;\eps_3,\eps_2)-\zeta_*(a,2;\eps_3,\eps_2\eps_1)+\zeta_*(1,a+1;\eps_1,\eps_3\eps_2) \\
&\, -a \zeta_*(1;\eps_1) \zeta_*(a+1;\eps_3)
 -\sum_{\mu+\nu=w}   \zeta_*(\nu,\mu;\eps_1,\eps_2)
+\zeta_*(a,2;\eps_3,\eps_2)  +a\zeta_*(a+1,1;\eps_3,\eps_2),
\end{align*}
we get
\begin{align}
&\, 2\zeta_*(1,1,a;\eps_1,\eps_2,\eps_3)
\equiv \zeta_*(a,2;\eps_3,\eps_1\eps_2)
-\zeta_*(1,a+1;\eps_1,\eps_2\eps_3) \notag\\
&\, -a \zeta_*(1;\eps_3) \zeta_*(a+1;\eps_1)
 -\sum_{\mu+\nu=w}   \zeta_*(\nu,\mu;\eps_1,\eps_2)
+\zeta_*(a,2;\eps_3,\eps_2)  +a\zeta_*(a+1,1;\eps_3,\eps_2)   \label{equ:zeta(1,1,a)}
\end{align}
which is exactly the right-hand side of the equation in Theorem~\ref{thm:PanzerDivES} when $(n_1,n_2,n_3)=(1,1,a)$
since $\zeta_*(a;\eps_3)\zeta_*(1,1;\eps_1,\eps_2)\equiv 0 \pmod{\zeta(2)}$.
This completes the proof of the theorem.
\end{proof}

\begin{cor}\label{cor:zeta(a,1,b-1);e1,e2,e3}
Suppose $a,b\in\N$ and $w=a+b\ge 4$ is even. Then
\begin{align}
&\,  2 \zeta_*(a,1,b-1;\eps_1,\eps_2,\eps_3) \equiv
 \zeta_*(b-1;\eps_3)\zeta(a+1; \eps_2)-\zeta(a+1,b-1; \eps_1\eps_2,\eps_3)-\zeta_*(a,b;\eps_1,\eps_2\eps_3) \notag\\
&\, -(b-1) \zeta_*(a;\eps_1) \zeta_*(b;\eps_3)
  -(-1)^{a}\sum_{s=1}^{w-1}\left[ \binom{s-1}{b-2} \zeta_*(s,w-s;\eps_3,\eps_2) -\binom{s-1}{a-1}\zeta_*(s,w-s;\eps_1,\eps_2)\right],  \label{equ:a1b-1ES}
\end{align}
and for even $a\ge 2$ and $w=a+2$
\begin{align}
  2 \zeta(a,1,1;\eps_1,\eps_2,\eps_3)
\equiv &\,\zeta(a,2;\eps_1,\eps_2)-\zeta(a,2;\eps_1,\eps_2\eps_3) -\zeta(a+1,1;\eps_1\eps_2,\eps_3) \notag \\
 &\,+a\zeta(a+1,1;\eps_1,\eps_2)+\zeta(a+1,1;\eps_2,\eps_3)  -\sum_{s=2}^{w-1}\zeta(s,w-s;\eps_3,\eps_2) . \label{equ:a11ES}
\end{align}
\end{cor}
\begin{proof}
These are special cases of Theorem~\ref{thm:PanzerDivES}. First, direct computations yields that
\begin{align*}
 2 \zeta_*(a,1,b-1;&\, \eps_1,\eps_2,\eps_3)
\equiv 2\zeta_*(b-1;\eps_3)\zeta_*(a,1;\eps_1,\eps_2)-\zeta_*(a,b;\eps_1,\eps_2\eps_3) \notag \\
&\, +\zeta_*(b-1,a+1;\eps_3,\eps_1\eps_2)-(b-1) \zeta_*(a;\eps_1) \zeta_*(b;\eps_3)-a\zeta(a+1;\eps_1) \zeta_*(b-1;\eps_3) \\
&\, -(-1)^{a}\sum_{\mu+\nu=w} \binom{\nu-1}{b-2} \zeta_
*(\nu,\mu;\eps_3,\eps_2)
+(-1)^a\sum_{\mu+\nu=w} \binom{\nu-1}{a-1} \zeta_
*(\nu,\mu;\eps_1,\eps_2)
\end{align*}
which can be reduced to \eqref{equ:a1b-1ES} readily using
\begin{equation*}
2\zeta(a,1;\eps_1,\eps_2)=a\zeta(a+1; \eps_1)+\zeta(a+1; \eps_2)-\zeta(a+1; \eps_1\eps_2)
\end{equation*}
from Prop.~\ref{prop:DBESreduce}. If further $b=2$, then  \eqref{equ:a1b-1ES} implies that
\begin{align*}
 2 \zeta_*(a,1,1&\, \eps_1,\eps_2,\eps_3)
\equiv \zeta_*(1;\eps_3)\zeta(a+1; \eps_2)-\zeta(a+1,1; \eps_1\eps_2,\eps_3)-\zeta_*(a,2;\eps_1,\eps_2\eps_3) \notag\\
&\, -\zeta_*(a;\eps_1) \zeta_*(2;\eps_3)
  -\sum_{s=1}^{w-1}\zeta_*(s,w-s;\eps_3,\eps_2) ,
+\zeta(a,2;\eps_1,\eps_2)+a\zeta(a+1,1;\eps_1,\eps_2).
\end{align*}
Therefore, \eqref{equ:a11ES} follows easily from the stuffle relation
\begin{equation*}
    \zeta_*(1;\eps_3)\zeta(a+1; \eps_2)\equiv \zeta_*(1,a+1;\eps_3, \eps_2)+ \zeta_*(a+1,1;\eps_2,\eps_3) \pmod{\zeta(2)}
\end{equation*}
and the fact that  $\zeta_*(2;\eps_3)\equiv 0\pmod{\zeta(2)}.$
\end{proof}

We now apply the above to compute the symmetric triple zeta values.
\begin{thm}\label{thm:SMZV(a,1,b-1;e1,e2,e3)}
Suppose $\eps_1, \eps_2,\eps_3 =\pm 1$, $a,b\in\N$ and $w=a+b\ge 4$ is even. Then, modulo $\zeta(2)$
\begin{align*}
2\zeta_\ast^\Sy(a,1,b-1; \eps_1,\eps_2,\eps_3)
 \equiv&\,    + (\eps_1+1)\zeta(a;\eps_1)\zeta(b;\eps_2) + (\eps_1 \eps_2+1)\zeta(a+1; \eps_2)\zeta_*(b-1;\eps_3)  \\
&\,+ (\eps_1-1) \Big[(b-1)\zeta(a;\eps_1)\zeta(b; \eps_3)+\zeta(a;\eps_1)\zeta(b;\eps_2\eps_3)\Big] \\
&\, + (\eps_1 \eps_2-1) \Big[ a \zeta(a+1; \eps_1)\zeta_*(b-1;\eps_3)+\zeta(a+1; \eps_1\eps_2)\zeta_*(b-1;\eps_3)\Big]\\
&\,    -2(-1)^{a}\sum_{s=1}^{w-1}\left[ \binom{s-1}{b-2} \zeta_*(s,w-s;\eps_3,\eps_2) -\binom{s-1}{a-1}\zeta_*(s,w-s;\eps_1,\eps_2)\right] .
\end{align*}
\end{thm}

\begin{proof}
By definition,
\begin{align*}
 &\, \zeta_\ast^\Sy(a,1,b-1;\eps_1,\eps_2,\eps_3)
\equiv \zeta_\ast(a,1,b-1;\eps_1,\eps_2,\eps_3)+(-1)^a \eps_1\zeta(a;\eps_1)\zeta_*(1,b-1;\eps_2,\eps_3)\\
&\,-(-1)^a \eps_1 \eps_2 \zeta_*(1,a;\eps_2,\eps_1)\zeta_*(b-1;\eps_3)
+\eps_1\eps_2\eps_3\zeta_\ast(b-1,1,a;\eps_3,\eps_2,\eps_1)\\
\equiv &\, \zeta_\ast(a,1,b-1;\eps_1,\eps_2,\eps_3)- \eps_1\zeta(a;\eps_1)\zeta_*(1,b-1;\eps_2,\eps_3)\\
&\,-\eps_1\eps_2\zeta_*(1,a;\eps_2,\eps_1)\zeta_*(b-1;\eps_3)+\eps_1\eps_2\eps_3\zeta_\ast(b-1,1,a;\eps_3,\eps_2,\eps_1)\\
\equiv &\, \zeta_\ast(a,1,b-1;\eps_1,\eps_2,\eps_3)-\eps_1\zeta(a;\eps_1)\zeta_*(1;\eps_2)\zeta_*(b-1;\eps_3)
+\eps_1\zeta(a;\eps_1)\zeta_*(b;\eps_2\eps_3)\\
&\,-\eps_1\zeta(a;\eps_1)\zeta_*(b-1,1;\eps_3,\eps_2)+\eps_1 \eps_2 \zeta(a,1;\eps_1,\eps_2)\zeta_*(b-1;\eps_3)- \eps_1 \eps_2 \zeta_*(1;\eps_2)\zeta(a;\eps_1)\zeta_*(b-1;\eps_3)\\
&\, +\eps_1 \eps_2 \zeta_*(a+1;\eps_1\eps_2)\zeta_*(b-1;\eps_3)
+\eps_1\eps_2\eps_3\zeta_\ast(b-1,1,a;\eps_3,\eps_2,\eps_1)\\
\equiv &\, \zeta_\ast(a,1,b-1;\eps_1,\eps_2,\eps_3)+\eps_1\zeta(a;\eps_1)\zeta_*(b-1,1;\eps_3,\eps_2)
+\eps_1\zeta(a;\eps_1)\zeta_*(b;\eps_2\eps_3)\\
&\,+\eps_1 \eps_2 \zeta(a,1;\eps_1,\eps_2)\zeta_*(b-1;\eps_3)
+\eps_1\eps_2\eps_3\zeta_\ast(b-1,1,a;\eps_3,\eps_2,\eps_1)
+\eps_1 \eps_2 \zeta_*(a+1;\eps_1\eps_2)\zeta_*(b-1;\eps_3).
\end{align*}
Now, plugging the expressions for the two triple Euler sums provided by Cor.~\ref{cor:zeta(a,1,b-1);e1,e2,e3} and using
\begin{equation*}
\zeta(a,1;\eps_1,\eps_2)=\frac12\Big[a\zeta(a+1; \eps_1)+\zeta(a+1; \eps_2)-\zeta(a+1; \eps_1\eps_2)\Big],
\end{equation*}
we get after simplification
 \begin{align*}
 &\, 2\zeta_\ast^\Sy(a,1,b-1;\eps_1,\eps_2,\eps_3)
\equiv
+ \eps_1 \zeta(a;\eps_1)\Big[(b-1)\zeta(b; \eps_3)+\zeta(b;\eps_2)+\zeta(b;\eps_2\eps_3)\Big]\\
&\,+ \eps_1 \eps_2  \zeta_*(b-1;\eps_3)\Big[a\zeta(a+1; \eps_1)+\zeta(a+1; \eps_2)+\zeta(a+1; \eps_1\eps_2)\Big]\\
&\, +\zeta_*(a;\eps_1)\zeta(b; \eps_2)-\zeta(a;\eps_1)\zeta(b;\eps_2\eps_3)-\zeta(a+1;\eps_2\eps_1)\zeta(b-1;\eps_3)\\
&\, -a\zeta(b-1;\eps_3) \zeta(a+1;\eps_1)+\zeta_*(b-1;\eps_3)\zeta(a+1; \eps_2)\\
&\, -(b-1) \zeta(a;\eps_1) \zeta(b;\eps_3)
  -2(-1)^{a}\sum_{s=1}^{w-1}\left[ \binom{s-1}{b-2} \zeta_*(s,w-s;\eps_3,\eps_2) -\binom{s-1}{a-1}\zeta_*(s,w-s;\eps_1,\eps_2)\right].
\end{align*}
This easily leads to the final expression in the theorem.
\end{proof}

The following corollary will be crucial for our computation of effective double Euler sums in the following sections.
\begin{cor}\label{cor:SMZV(a,1,b-1;eps,1,eps)}
Let $\eps=-1$ and $a,b\in\N$ with even $w=a+b\ge 4$. Then, modulo $\zeta(2)$
\begin{align*}
&\, 2\zeta_\ast^\Sy(a,1,b-1; \eps,1,\eps)
\equiv-2\zeta_*(a;\eps)\zeta(b;\eps)
+(\eps+1)\Big[\zeta(a;\eps)\zeta(b)
+\zeta(a+1)\zeta_*(b-1;\eps) \Big]\\
&\,+(\eps-1)\Big[
b\zeta(a;\eps)\zeta(b; \eps) +(a+1)\zeta(a+1; \eps)\zeta_*(b-1;\eps) \Big]
 -2(-1)^{a}\sum_{s=1}^{w-1}\binom{s}{b-1}\zeta_*(s,w-s;\eps,1) .
\end{align*}
\end{cor}

\begin{proof}
Taking $\eps_1=\eps_3=\eps$ and $\eps_2=1$ in Theorem~\ref{thm:SMZV(a,1,b-1;e1,e2,e3)}, we obtain
\begin{align*}
&\, 2\zeta_\ast^\Sy(a,1,b-1;\eps,1,\eps)
\equiv(\eps-1) (b-1)\zeta(a;\eps)\zeta(b; \eps)
+(\eps+1)\zeta(a;\eps)\zeta(b)
+(\eps-1)\zeta(a;\eps)\zeta(b;\eps)\\
&\,+(\eps-1)a\zeta(a+1; \eps)\zeta_*(b-1;\eps)
+(\eps+1)\zeta(a+1)\zeta_*(b-1;\eps)
+(\eps-1)\zeta(a+1;\eps)\zeta(b-1;\eps) \\
&\, -2(-1)^{a}\sum_{s=1}^{w-1}\left[ \binom{s-1}{b-2} \zeta_*(s,w-s;\eps,1) -\binom{s-1}{a-1}\zeta_*(s,w-s;\eps,1)\right].
\end{align*}
The corollary follows immediately from  the shuffle product formula
\begin{align*}
\zeta_*(a;\eps)\zeta_*(b;\eps)=\zeta_\sha(a;\eps)\zeta_\sha(b;\eps)= &\, \sum_{s=1}^{w-1} \bigg[\binom{s-1}{a-1}+\binom{s-1}{b-1}\bigg] \zeta_*(s,w-s;\eps,1)
\end{align*}
and the binomial identity $\binom{s-1}{b-2}+\binom{s-1}{b-1}=\binom{s}{b-1}$.
\end{proof}

\section{Explicit Effective Double Euler Sums of Even Weight}\label{sec:depthTwoEvenWt}
By using their own heuristic definition of effective double zeta values $Z_\calA(s,w-s)$ in Definitionf~\ref{defn:KZdefn}, Kaneko and Zagier made the following conjecture.

\begin{conj}\label{conj:KZeffDBZs} \emph{(\cite[(65)]{KanekoZa2026})}
Let $a,b\in\N$ such that $w=a+b\ge 4$ is even. Set
\begin{equation*}
\Psi_s:=\zeta_\calA(s,1,w-s-1)-Z_\calA(s)Z_\calA(w-s) \quad (1\le s\le w-1).
\end{equation*}
Then
\begin{align} \label{equ:KZrel1}
&\, (-1)^a \sum_{s=a}^{w-1} \binom{s}{a} Z_\calA(s,w-s) =\Psi_a
     \quad (1\le a\le w-2),\\
&\, Z_\calA(w-1,1)+\sum_{s=1}^{w-1} Z_\calA(s,w-s) =0, \label{equ:KZrel2} \\
&\, Z_\calA(a,b)\equiv(-1)^a \sum_{s=1}^{w-2} \left[ \binom{s}{a}-\frac12\binom{w-1}{a}\right] \Psi_s   \quad (1\le a\le w-1). \label{equ:KZrel3}
\end{align}
\end{conj}

\subsection{A lemma on linear algebra}
Our main idea to tackle Conjecture~\ref{conj:KZeffDBZs} (and more generally on double Euler sums) is motivated by Yasuda's theorem \cite{Yasuda2016} which says that the symmetric MZVs generate the whole space of MZVs. Yasuda's result is extended to higher level analogs, including the Euler sums, by Anzawa \cite{Anzawa2024}.  We will find a system  of linear equations relating double Euler sums and a special type of symmetric triple Euler sums of the form $\zeta_\sha^\Sy(a,1,b-1;\eps_1,\eps_2,\eps_3)$ of even weight. Then we can find a solution for all the double Euler sums  of even weight in terms of such triple values and (products of) single zeta values. In this process, we will need the following linear algebra result.

\begin{lem}\label{lem:linAlg}
Let $w\ge 4$ be an even number. Let $x_s$ ($1\le s\le w-1$) be formal variables satisfying the following system of equations:
\begin{align}\label{equ:FormalRel1}
\sum_{s=a}^{w-1}  \binom{s}{a} x_s=& \, (-1)^a \psi_a \quad (1\le a\le w-2), \\
 -\sum_{s=1}^{w-2} x_s=& \,A x_{w-1}+B, \label{equ:FormalRel2}
\end{align}
where  $A$, $B$ and $\psi_a$ are constants with $A\ne 0$. Then,
\begin{align}\label{equ:FormalRel3}
x_a=& \, (-1)^a \sum_{s=1}^{w-2} \left[ \binom{s}{a}-\frac1{A}\binom{w-1}{a}\right]  \psi_s+(-1)^a \frac{B}{A} \binom{w-1}{a}  \quad (1\le a\le w-2), \\
x_{w-1}=& \, \frac{-1}{A}\left(B+\sum_{s=1}^{w-2} x_s\right). \label{equ:FormalRel4}
\end{align}
Conversely, if \eqref{equ:FormalRel3} and \eqref{equ:FormalRel4} hold then the variables $x_s$ ($1\le s\le w-1$) satisfy both \eqref{equ:FormalRel1} and \eqref{equ:FormalRel2}.
\end{lem}
\begin{proof}
From \eqref{equ:FormalRel1} we get
\begin{align*}
     \sum_{s=1}^{w-2} \psi_s =&\,  \sum_{s=1}^{w-2} (-1)^s \sum_{t=s}^{w-1}  \binom{t}{s} x_t \\
     =&\, \sum_{s=1}^{w-2}(-1)^s  \binom{w-1}{s} x_{w-1} +  \sum_{t=1}^{w-2}   \left[ \sum_{s=1}^{t}  (-1)^s\binom{s}{a}\binom{t}{s} \right]  x_t
    =  - \sum_{t=1}^{w-2} x_t =A x_{w-1}+B
\end{align*}
by \eqref{equ:FormalRel2}. Hence,
\begin{align*}
&\,\text{RHS of } \eqref{equ:FormalRel3}+  (-1)^a \binom{w-1}{a}   x_{w-1}\\
\equiv &\,
\text{RHS of }  \eqref{equ:FormalRel3}+\frac{(-1)^a}{A} \binom{w-1}{a} \bigg(\sum_{s=1}^{w-2} \psi_s -B\bigg)\\
\equiv &\,(-1)^a \sum_{s=a}^{w-2} \binom{s}{a} \psi_s  \\
\equiv &\,(-1)^a \sum_{s=a}^{w-2}   \binom{s}{a}  (-1)^s \sum_{t=s}^{w-1}  \binom{t}{s} x_t  \\
\equiv &\,(-1)^a \sum_{s=a}^{w-2}  (-1)^s   \binom{s}{a}   \binom{w-1}{s} x_{w-1}
+(-1)^a  \sum_{t=a}^{w-2}  \sum_{s=a}^{t}  (-1)^s \binom{s}{a} \binom{t}{s} x_t.
\end{align*}
Note that
\begin{equation*}
\sum_{s=a}^{t}  (-1)^s\binom{s}{a}\binom{t}{s}
= \frac{t!}{(t-a)!} \sum_{s=0}^{t-a} \frac{(-1)^{s+a} (t-a)!}{(t-a-s)! a!}
= \delta_{a,t} (-1)^a
\end{equation*}
where $\delta_{a,t}$ is the Kronecker symbol. In particular,
\begin{equation*}
\sum_{s=1}^{w-2}  (-1)^s\binom{s}{a}\binom{w-1}{s}
=\binom{w-1}{a}+(-1)^a\delta_{a,w-1}  .
\end{equation*}
Hence,
\begin{align*}
 &\,  \text{RHS of } \eqref{equ:FormalRel3} +(-1)^a \binom{w-1}{a} x_{w-1} \\
= &\,  (-1)^a \binom{w-1}{a} x_{w-1}+ \delta_{a,w-1}  \cdot x_{w-1}+  (1-\delta_{a,w-1}) \cdot x_a  \\
= &\,   (-1)^a \binom{w-1}{a} x_{w-1}+ x_a .
\end{align*}
This yields \eqref{equ:FormalRel3} immediately. The converse statement is a similar exercise on manipulating binomial sums and is thus left to the interested reader.
\end{proof}

\subsection{Effective double zeta values $Z_\calA^\eff(a,b)$}
We now show that Conjecture~\ref{conj:KZeffDBZs} holds if we replace $Z_\calA$ by the effective double zeta values $Z_\calA^\eff$.
We first treat \eqref{equ:KZrel1} and \eqref{equ:KZrel2} in Theorem~\ref{thm:KZeffDBZs} and then derive \eqref{equ:KZrel3}
in Corollary~\ref{cor:KZeffDBZs}.

\begin{thm}\label{thm:KZeffDBZs}
Assume $b\ge 2$ and $w=a+b\ge 4$ is even. Then, modulo $\zeta(2)$ we have
\begin{align}\label{equ:KZrel1SMZversion}
&\zeta_\ast^\Sy(a,1,b-1)-\zeta_*(a)\zeta(b)
\equiv  (-1)^a \sum_{s=a}^{w-1}  \binom{s}{a}  \zeta_*(s,w-s), \\
&\zeta(w-1,1)+\sum_{s=1}^{w-1} \zeta_*(s,w-s)\equiv 0. \label{equ:KZrel2SMZversion}
\end{align}
\end{thm}
\begin{proof}
First,  the shuffle product
\begin{equation*}
   \tb\sha \ta^{w-2}  \tb=\ta^{w-2} \tb^2 +\sum_{s=1}^{w-1} \ta^{s-1} \tb \ta^{w-s-1} \tb
\end{equation*}
we see that
\begin{equation*}
    0=\zeta_*(1)\zeta(w-1)=\zeta(w-1,1)+\sum_{s=1}^{w-1} \zeta_*(s,w-s)
\end{equation*}
which proves \eqref{equ:KZrel2SMZversion}.

Next, taking $\eps=1$ in Cor.~\ref{cor:SMZV(a,1,b-1;eps,1,eps)} and applying the substitution $(a,b)\to (b-1,a+1)$ we get
\begin{equation}\label{equ:checkb=2}
\zeta_\ast^\Sy(a,1,b-1)-\zeta_*(a)\zeta(b)
\equiv   (-1)^a \sum_{s=a}^{w-1} \binom{s}{a}  \zeta_*(s,w-s)
\end{equation}
since $\zeta_\ast^\Sy(b-1,1,a)=\zeta_\ast^\Sy(a,1,b-1)$ and $a$ and $b$ have the same parity.  This gives \eqref{equ:KZrel1SMZversion} and thus completes the proof of the theorem.
\end{proof}

\begin{cor}\label{cor:KZeffDBZs}
Suppose $a$ and $b$ are positive integers with even $w=a+b\ge 4$. If $1\le a\le w-2$ then we have
\begin{align} \label{equ:KZrel3SMZversion}
&\, \zeta_*(a,b)\equiv (-1)^a \sum_{s=1}^{w-2} \left[ \binom{s}{a}-\frac12\binom{w-1}{a}\right] \Big(\zeta_\ast^\Sy(s,1,w-s-1)-\zeta_*(s)\zeta(w-s)\Big),\\
&\, Z_\calA^\eff(a,b)=(-1)^a \sum_{s=1}^{w-2} \left[ \binom{s}{a}-\frac12\binom{w-1}{a}\right] \Big(\zeta_\calA(s,1,w-s-1)-Z_\calA^\eff(s)Z_\calA^\eff(w-s)\Big). \label{equ:KZrel3FMZversion}
\end{align}
Further,
\begin{align} \label{equ:KZrel3FMZversion2}
Z_\calA^\eff(w-1,1)=-\frac12 \sum_{a=1}^{w-2}Z_\calA^\eff(a,b).
\end{align}
\end{cor}

\begin{proof}
We only need to prove \eqref{equ:KZrel3SMZversion} which follows from Lemma~\ref{lem:linAlg} by taking $A=2$, $B=0$,
$x_s= \zeta_*(s,w-s)$  and $\psi_a=\zeta_\ast^\Sy(a,1,w-a-1)-\zeta_*(a)\zeta(w-a)$.
\end{proof}

\begin{lem}\label{lem:triple=0}
Let $w\ge 4$ be an even number. Then
\begin{equation*}
\sum_{s=1}^{w-2} \zeta_\calA(s,1,w-s-1)=0.
\end{equation*}
\end{lem}

\begin{proof}
By taking $k=w$, $n=3$ and $i=2$ in \cite[Theorem 1.4]{SaitoWa2013b}, we get
\begin{equation*}
\sum_{a+b+c=w,b\ge 2} \zeta_\calA(a,b,c)=0.
\end{equation*}
But for all $a,b,c\in\N$, we have
\begin{align*}
0=\sum_{a+b+c=w}  \zeta_\calA(a,b)  \zeta_\calA(c) & =3 \sum_{a+b+c=w}  \zeta_\calA(a,b,c)
 + \sum_{a+b+c=w}\Big( \zeta_\calA(a,b+c) +\zeta_\calA(a+c,b) \Big)\\
 & =3 \sum_{a+b+c=w}  \zeta_\calA(a,b,c)=0.
\end{align*}
Hence
\begin{align*}
\sum_{s=1}^{w-2} \zeta_\calA(s,1,w-s-1)
&=\sum_{a+b+c=w}  \zeta_\calA(a,b,c)-\sum_{a+b+c=w,b\ge 2} \zeta_\calA(a,b,c)=0
\end{align*}
as desired.
\end{proof}

\begin{prop}\label{prop:doubleZetaProperties}
Let $w\ge 4$ be an even integer. Then,  we have
\begin{enumerate}
  \item[\upshape{(i)}] Stuffle relation: $Z_\calA^\eff(a)Z_\calA^\eff(w-a)= Z_\calA^\eff(a,w-a)+ Z_\calA^\eff(w-a,a)$ for all $1\le a\le w-1$;

  \item[\upshape{(ii)}] Sum formula: $\displaystyle \sum_{a=2}^{w-1} Z_\calA^\eff(a,w-a)=0$;

  \item[\upshape{(iii)}] Weighted sum formula: $\displaystyle \sum_{a=2}^{w-1} 2^a Z_\calA^\eff(a,w-a)=0$.
\end{enumerate}
\end{prop}

\begin{rem}
When the weight is odd, the three corresponding formulas are given in Prop.~\ref{prop:StuffleRel}, Prop.~\ref{prop:SumFormula}, and \eqref{equ:OddWtSum(a,b)}.
\end{rem}

\begin{proof}
(i). We first observe that
\begin{align}\label{equ:CoefSumId}
\sum_{a=2}^{w-2}(-1)^a
\left[\binom{s}{a}-\frac12\binom{w-1}{a}\right]
=-\frac 12+\left(\binom{s}{1}-\frac12\binom{w-1}{1}\right)+\frac 12(-1)^{w-1}=s-\frac 12-\frac w2.
\end{align}
Set $\Psi_s =\zeta_\calA(s,1,w-s-1)-Z_\calA^\eff(s)Z_\calA^\eff(w-s)$ for $1\le s\le w-2$. Let
\begin{align*}
X_a:=&  Z_\calA^\eff(a,w-a)+ Z_\calA^\eff(w-a,a).
\end{align*}
Then we have
\begin{align*}
2X_1=&\, Z_\calA^\eff(1,w-1)-\sum_{a=2}^{w-2}Z_\calA^\eff(a,w-a) \\
 =&\,-\sum_{s=1}^{w-2} \left[ s-\frac{w-1}2 +\sum_{a=2}^{w-2} (-1)^a \left( \binom{s}{a}-\frac12\binom{w-1}{a}\right)\right] \Psi_s
 =-\sum_{s=1}^{w-2} \big(2s-w\big) \Psi_s.
\end{align*}
Note that $Z_\calA^\eff(1)=0$. By applying the substitution $s\to w-s-1$ for the sum involving triple finite MZVs,
the substitution $s\to w-s$ for the sum involving product of effective Riemann zetas,
and the reversal relation $\zeta_\calA(s,1,w-s-1)=\zeta_\calA(w-s-1,1,s)$, we get
\begin{align*}
4X_1 =&\, -\sum_{s=1}^{w-2} \big(2s-w+2(w-s-1)-w\big)\zeta_\calA(s,1,w-s-1) \\
&\, +\sum_{s=1}^{w-2} \big(2s-w+2(w-s)-2\big) Z_\calA^\eff(s)Z_\calA^\eff(w-s)\Big) \\
=&\, 2\sum_{s=1}^{w-2} \zeta_\calA(s,1,w-s-1)=0.
\end{align*}
by Lemma~\ref{lem:triple=0}. This proves (i) when $a=1$ or $a=w-1$.

Now we assume $2\le a\le w-2$. Then
\begin{align*}
X_a=&\,   (-1)^a \sum_{s=1}^{w-2} \left[ \binom{s}{a}-\frac12\binom{w-1}{a}+\binom{s}{w-a}-\frac12\binom{w-1}{w-a}\right]\Psi_s\\
=&\,   (-1)^a \sum_{s=1}^{w-2} \left[ \binom{s}{a}+\binom{s}{w-a}-\frac12\binom{w}{a}\right] \Big(\zeta_\calA(s,1,w-s-1)-Z_\calA^\eff(s)Z_\calA^\eff(w-s)\Big).
\end{align*}
By Lemma~\ref{lem:triple=0} and \cite[Cor.~5.9]{BachmannRisan2026},
\begin{align*}
&\, \sum_{s=1}^{w-2} \left[ \binom{s}{a}+\binom{s}{w-a}-\frac12\binom{w}{a}\right]\zeta_\calA(s,1,w-s-1)\\
=&\, -\frac12  \sum_{s=2}^{w-2}  \sum_{\substack{i+j=a\\ i<\min\{a,s\}\\ j<\min\{a,w-s\} }} \binom{s}{i}\binom{w-s}{j}  Z_\calA^\eff(s)Z_\calA^\eff(w-s)\\
=&\, \frac12  \Bigg\{\sum_{s=2}^{w-2}  \left[\binom{s}{a}+\binom{s}{w-a}+\binom{w-s}{a}+\binom{w-s}{w-a} - \binom{w}{a}\right] Z_\calA^\eff(s)Z_\calA^\eff(w-s) \\
&\, - \binom{a}{a}Z_\calA^\eff(a)Z_\calA^\eff(w-a)-\binom{w-a}{w-a}Z_\calA^\eff(a)Z_\calA^\eff(w-a)\Bigg\}\\
=&\,  \left\{\sum_{s=2}^{w-2}  \left[\binom{s}{a}+\binom{s}{w-a} - \frac12 \binom{w}{a}\right] Z_\calA^\eff(s)Z_\calA^\eff(w-s)\right\}- Z_\calA^\eff(a)Z_\calA^\eff(w-a)
\end{align*}
by the substitution $s\to w-s$, we see that
\begin{align*}
X_a=& \, -(-1)^a Z_\calA^\eff(a)Z_\calA^\eff(w-a)=Z_\calA^\eff(a)Z_\calA^\eff(w-a)
\end{align*}
since $Z_\calA^\eff(a)=0$ if $a$ is even. This completes the proof of (i) for all $1\le a\le w-1$.

\bigskip
(ii). By \eqref{equ:KZrel3FMZversion2},
\begin{align*}
2 \sum_{a=2}^{w-1} Z_\calA^\eff(a,w-a)& = -Z_\calA^\eff(1,w-1)+ \sum_{a=2}^{w-2} Z_\calA^\eff(a,w-a)\\
& = \sum_{a=2}^{w-2} (-1)^a \sum_{s=1}^{w-2} \left[ \binom{s}{a}-\frac12\binom{w-1}{a}\right] \Psi_s
+\sum_{s=1}^{w-2} \left[ \binom{s}{1}-\frac12\binom{w-1}{1}\right] \Psi_s.
\end{align*}
Hence,
\begin{align*}
2 \sum_{a=2}^{w-1} Z_\calA^\eff(a,w-a)
& = \sum_{s=1}^{w-2}(2s-w) \Psi_s.
\end{align*}
Note that $Z_\calA^\eff(1)=0$. By applying the same substitutions as in the proof of (i) and the reversal relation $\zeta_\calA(s,1,w-s-1)=\zeta_\calA(w-s-1,1,s)$, we get
\begin{align*}
4\sum_{a=2}^{w-1} Z_\calA^\eff(a,w-a)
& =  \sum_{s=1}^{w-2}\big(2s-w\big)\zeta_\calA(s,1,w-s-1)
- \sum_{s=2}^{w-2}\big(2s-w\big) Z_\calA^\eff(s)Z_\calA^\eff(w-s)\\
&+ \sum_{s=1}^{w-2}\big(2(w-s-1)-w\big) \zeta_\calA(s,1,w-s-1)
-\sum_{s=2}^{w-2}\big(2(w-s)-w\big) Z_\calA^\eff(s)Z_\calA^\eff(w-s)\\
&=-2\sum_{s=1}^{w-2} \zeta_\calA(s,1,w-s-1)=0
\end{align*}
by Lemma~\ref{lem:triple=0}. This proves (ii).

\bigskip
(iii). By \eqref{equ:KZrel3FMZversion2},
\begin{align*}
2\sum_{a=2}^{w-1} 2^a Z_\calA^\eff(a,w-a)& = \sum_{a=2}^{w-2} (2^{a+1} -2^w)Z_\calA^\eff(a,w-a)\\
& = \sum_{a=2}^{w-2} (-1)^a (2^{a+1} -2^w) \sum_{s=1}^{w-2} \left[ \binom{s}{a}-\frac12\binom{w-1}{a}\right] \Psi_s.
\end{align*}
Since
\begin{align*}
\sum_{a=2}^{w-2}(-1)^a  2^{a+1}
\left[\binom{s}{a}-\frac12\binom{w-1}{a}\right]
& =2(-1)^s-(-1)^{w-1}+2\left(-\frac 12+2\left(\binom{s}{1}-\frac12\binom{w-1}{1}\right)+\frac 12(-2)^{w-1}\right) \\
&=2(-1)^s+4s+2-2w-2^{w-1},
\end{align*}
and
\begin{align*}
 \sum_{a=2}^{w-2}(-1)^a  2^w
\left[\binom{s}{a}-\frac12\binom{w-1}{a}\right]
= 2^w \left(s-\frac 12-\frac w2\right),
\end{align*}
we find that
\begin{align*}
2\sum_{a=2}^{w-1} 2^a Z_\calA^\eff(a,w-a)
& = \sum_{s=1}^{w-2}\Big(2(-1)^s+(4-2^w)s+2+(2^{w-1}-2)w\Big) \Psi_s.
\end{align*}
By applying the same substitutions as above, we get
\begin{align*}
8\sum_{a=2}^{w-1} 2^a  Z_\calA^\eff(a,w-a)
& =  \sum_{s=1}^{w-2}\Big(2(-1)^s+(4-2^w)s+2+(2^{w-1}-2)w\Big) \zeta_\calA(s,1,w-s-1) \\
&+ \sum_{s=1}^{w-2}\Big(-2(-1)^s+(4-2^w)(w-s-1)+2+(2^{w-1}-2)w\Big) \zeta_\calA(s,1,w-s-1)\\
&- \sum_{s=2}^{w-2}\Big(2(-1)^s+(4-2^w)s+2+(2^{w-1}-2)w\Big)  Z_\calA^\eff(s)Z_\calA^\eff(w-s)\\
&-\sum_{s=2}^{w-2}\Big(2(-1)^s+(4-2^w)(w-s)+2+(2^{w-1}-2)w\Big) Z_\calA^\eff(s)Z_\calA^\eff(w-s)\\
&= 2^w\sum_{s=1}^{w-2}\zeta_\calA(s,1,w-s-1) -4\sum_{s=2}^{w-2} ((-1)^s+1) Z_\calA^\eff(s)Z_\calA^\eff(w-s)=0
\end{align*}
since $Z_\calA^\eff(s)=0$ for all even $s$. This concludes the proof of (iii) and the proposition.
\end{proof}

\subsection{Effective Euler sums $Z_\calA^\eff(\ol{a},b)$ and $Z_\calA^\eff(a,\ol{b})$ of even weight}
We will construct the effective double Euler sums of the form $Z_\calA^\eff(\ol{a},b)$ and $Z_\calA^\eff(a,\ol{b})$ of even weight in this section. As before, we first express some special symmetric triple Euler sums of even weight in terms of double and product of single Euler sums by the parity principle.

\begin{thm}\label{thm:effDBES(-1,1)}
Assume $b\ge 2$ and $w=a+b$ is even. Then, modulo $\zeta(2)$ we have
\begin{align}\label{equ:KZrel1SMES(-1,1)version}
&\, \zeta_\ast^\Sy(\ol{a},1,\ol{b-1})
    \equiv -b\zeta(\ol{a})\zeta(\ol{b})
            -(a+2)\zeta(\ol{a+1})\zeta(\ol{b-1})
            +(-1)^{a}\sum_{s=1}^{w-1}\binom{s}{a}\zeta_*(\ol{s},w-s),\\
&\zeta(\ol{w-1},1)+\sum_{s=1}^{w-1} \zeta(\ol{s},w-s)\equiv \zeta(\ol{1})\zeta(\ol{w-1}).\label{equ:KZrel2SMES(-1,1)version}
\end{align}
\end{thm}

\begin{proof}
Taking $\eps=-1$ in Cor.~\ref{cor:SMZV(a,1,b-1;eps,1,eps)} we get
\begin{align*}
&\, \zeta_\ast^\Sy(\ol{a},1,\ol{b-1})
\equiv -(b+1)\zeta(\ol{a})\zeta(\ol{b})
-(a+1)\zeta(\ol{a+1})\zeta_*(\ol{b-1})
-(-1)^{a}\sum_{s=1}^{w-1}\binom{s}{b-1}\zeta_*(\ol{s},w-s) .
\end{align*}
Since $\zeta_\ast^\Sy(\ol{b-1},1,\ol{a})=\zeta_\ast^\Sy(\ol{a},1,\ol{b-1})$, applying the substitution $(a,b)\to (b-1,a+1)$, we can immediately obtain \eqref{equ:KZrel1SMES(-1,1)version}.

Finally, setting $\tc=dt/(1+t)$ we see that \eqref{equ:KZrel2SMES(-1,1)version} follows easily from the shuffle product
\begin{equation*}
    \tc\sha \ta^{w-2} \tc=\ta^{w-2} \tc^2 +\sum_{s=1}^{w-1} \ta^{s-1} \tc \ta^{w-s-1} \tc.
\end{equation*}
This completes the proof of the theorem.
\end{proof}

\begin{cor}\label{cor:KZeffDBES(-1,1)}
Suppose $a,b$ are positive integers with even $w=a+b\ge 4$. If $1\le a\le w-2$ then we have
\begin{align}
&\zeta(\ol{a},b)\equiv
 \sum_{s=1}^{w-2} \left[ \binom{s}{a}-\frac12\binom{w-1}{a}\right]
\left\{\aligned
&\zeta_\ast^\Sy(\ol{s},1,\ol{r-1})+r\zeta(\ol{s})\zeta(\ol{r})\\
&+(s+2)\zeta(\ol{s+1})\zeta(\ol{r-1})
\endaligned \right\}
-\frac{(-1)^a }2\binom{w-1}{a}\zeta(\ol{1})\zeta(\ol{w-1}) ,  \label{equ:KZrel3SMES(-1,1)}\\
&Z_\calA^\eff(\ol{a},b)=
 (-1)^a \sum_{s=1}^{w-2} \left[ \binom{s}{a}-\frac12\binom{w-1}{a}\right]
\left\{\aligned
&\zeta_\calA(\ol{s},1,\ol{r-1})+rZ_\calA^\eff(\ol{s})Z_\calA^\eff(\ol{r})\\
&+(s+2)Z_\calA^\eff(\ol{s+1})Z_\calA^\eff(\ol{r-1})
\endaligned \right\} \notag\\
&\hskip9cm-\frac{(-1)^a }2\binom{w-1}{a}Z_\calA^\eff(\ol{1})Z_\calA^\eff(\ol{w-1}) ,  \label{equ:KZrel3FES(-1,1)}
\end{align}
where $r=w-s$, and
\begin{align} \label{equ:KZrel1SMES(1,-1)}
\zeta_\ast(a,\ol{b})\equiv &\, \zeta_\ast(a)\zeta(\ol{b})-\zeta(\ol{b},a),\\
Z_\calA^\eff(a,\ol{b})\equiv &\, Z_\calA^\eff(a)\zeta_\calA(\ol{b})-Z_\calA^\eff(\ol{b},a).\label{equ:KZrel2FES(1,-1)}
\end{align}
Further,
\begin{equation} \label{equ:KZrel3SMES(1,-1)}
Z_\calA^\eff(\ol{w-1},1) = \frac12\left( Z_\calA^\eff(\ol{1})Z_\calA^\eff(\ol{w-1})-\sum_{s=1}^{w-2}  Z_\calA^\eff(\ol{s},w-s)\right).
\end{equation}
Additionally, for all $1\le a\le w-1$ we have
\begin{align} \label{equ:KZrel3SMES(1,-1)2}
Z_\calA^\eff(a,\ol{b})\equiv &\, Z_\calA^\eff(a)Z_\calA^\eff(\ol{b})-Z_\calA^\eff(\ol{b},a).
\end{align}
\end{cor}

\begin{proof}
We only need to prove \eqref{equ:KZrel3SMES(-1,1)} which follows from Lemma~\ref{lem:linAlg} by taking $A=2$, $B=-\zeta(\ol{1})\zeta(\ol{w-1}) $,
$x_s=\zeta_*(\ol{s},w-s)$ and
\begin{align*}
\psi_a=\zeta_\ast^\Sy(\ol{a},1,\ol{b-1})
   +b\zeta(\ol{a})\zeta(\ol{b})
            +(a+2)\zeta(\ol{a+1})\zeta(\ol{b-1}) \quad \text{for}\quad 1\le a\le w-2,
\end{align*}
where $b=w-a$.
\end{proof}

\begin{eg} For any even weight $w$, we have
\begin{align}
Z_\calA^\eff(\ol{1},w-1)
=&\, \frac12(w-1)Z_\calA^\eff(\ol{1})Z_\calA^\eff(\ol{w-1}) \notag \\
&\,-\sum_{s=1}^{w-2} \left( s-\frac{w-1}2 \right)
\left\{\aligned
&\zeta_\calA(\ol{s},1,\ol{w-s-1})+(w-s)Z_\calA^\eff(\ol{s})Z_\calA^\eff(\ol{w-s})\\
&+(s+2)Z_\calA^\eff(\ol{s+1})Z_\calA^\eff(\ol{w-s-1})\endaligned \right\}.\label{equ::KZrel3SMES(1,-1)eg}
\end{align}
\end{eg}

\begin{prop}
Let $w\ge 4$ be an even integer. Then we have
\begin{align}
\sum_{a=2}^{w-1} 2^a Z_\calA^\eff(\ol{a},w-a)=&\,0. \label{equ:ConjWtSumFormulaFDESbarbar}
\end{align}
\end{prop}

\begin{rem}
The weighted sum formula for double Euler sums corresponding to \eqref{equ:ConjWtSumFormulaFDESbarbar} is given in the proof of \cite[Theorem 4.4]{BCJXXZhao2020c}. The odd weight case is given by
\end{rem}

\begin{proof}
By definition, setting $r=w-s$, we get
\begin{align*}
&\,4\sum_{a=2}^{w-1} 2^a Z_\calA^\eff(\ol{a},w-a)=
\sum_{a=2}^{w-2} (2^{a+2}-2^w)  Z_\calA^\eff(\ol{a},w-a)+2^{w}Z_\calA^\eff(\ol{1})Z_\calA^\eff(\ol{w-1})  -2^w Z_\calA^\eff(\ol{1},w-1)\\
=&\, -Z_\calA^\eff(\ol{1})Z_\calA^\eff(\ol{w-1}) \left(2^{w-1}(w-3)+\sum_{a=2}^{w-2}(2^{a+2}-2^{w}) \frac{(-1)^a}2 \binom{w-1}{a}\right) \\
&\, + \sum_{a=2}^{w-2} (2^{a+2}-2^w) (-1)^a \sum_{s=1}^{w-2} \left[ \binom{s}{a}-\frac12\binom{w-1}{a}\right]
\left\{\aligned
&\zeta_\calA(\ol{s},1,\ol{r-1})+rZ_\calA^\eff(\ol{s})Z_\calA^\eff(\ol{r})\\
&+(s+2)Z_\calA^\eff(\ol{s+1})Z_\calA^\eff(\ol{r-1})
\endaligned \right\} \\
&\, +2^w\sum_{s=1}^{w-2} \left( s-\frac{w-1}2 \right)
\left\{\aligned
&\zeta_\calA(\ol{s},1,\ol{r-1})+rZ_\calA^\eff(\ol{s})Z_\calA^\eff(\ol{r})\\
&+(s+2)Z_\calA^\eff(\ol{s+1})Z_\calA^\eff(\ol{r-1})\endaligned \right\}.
\end{align*}
Since
\begin{align*}
\sum_{a=2}^{w-2} (2^{a+2}-2^{w})(-1)^a \binom{w-1}{a}= 8w-16+(3-w)2^w
\end{align*}
and
\begin{align*}
\sum_{a=2}^{s} (2^{a+2}-2^{w})(-1)^a \binom{s}{a}= 4(-1)^s+8s-4-(s-1)2^w,
\end{align*}
we get
\begin{align*}
&\,4\sum_{a=2}^{w-1} 2^a Z_\calA^\eff(\ol{a},w-a) \\
=&\, -Z_\calA^\eff(\ol{1})Z_\calA^\eff(\ol{w-1}) \Big(4w-8\Big) \\
&\, +   \sum_{s=1}^{w-2} \Big(4(-1)^s+8s-4w+4+(w-2s-1)2^{w-1} \Big)
\left\{\aligned
&\zeta_\calA(\ol{s},1,\ol{r-1})+rZ_\calA^\eff(\ol{s})Z_\calA^\eff(\ol{r})\\
&+(s+2)Z_\calA^\eff(\ol{s+1})Z_\calA^\eff(\ol{r-1})
\endaligned \right\} \\
&\, +2^w\sum_{s=1}^{w-2} \left( s-\frac{w-1}2 \right)
\left\{\aligned
&\zeta_\calA(\ol{s},1,\ol{r-1})+rZ_\calA^\eff(\ol{s})Z_\calA^\eff(\ol{r})\\
&+(s+2)Z_\calA^\eff(\ol{s+1})Z_\calA^\eff(\ol{r-1})\endaligned \right\}\\
=&\, -Z_\calA^\eff(\ol{1})Z_\calA^\eff(\ol{w-1}) \Big(4w-8\Big)
  +   \sum_{s=1}^{w-2} \Big(4(-1)^s+8s-4w+4 \Big)
\zeta_\calA(\ol{s},1,\ol{r-1}) \\
&\, +   \sum_{s=1}^{w-2} \Big(4(-1)^s+8s-4w+4 \Big) rZ_\calA^\eff(\ol{s})Z_\calA^\eff(\ol{r})\\
& +  \sum_{s=2}^{w-1} \Big(-4(-1)^s+8s-4w-4\Big) (s+1)Z_\calA^\eff(\ol{s})Z_\calA^\eff(\ol{r})\\
=&\,  \sum_{s=1}^{w-2} \Big(4(-1)^s+8s-4w+4 \Big) \zeta_\calA(\ol{s},1,\ol{r-1})
+  (w+1) \sum_{s=2}^{w-2} \Big(8s-4w\Big)  Z_\calA^\eff(\ol{s})Z_\calA^\eff(\ol{r})
\end{align*}
since we $Z_\calA^\eff(\ol{s})=$ for all even $s$. Now, by the substitution $s\to w-s-1$  we see that
\begin{equation*}
 \sum_{s=1}^{w-2} \Big(4(-1)^s+8s-4w+4 \Big) \zeta_\calA(\ol{s},1,\ol{r-1}) = 0,
\end{equation*}
and by the  substitution $s\to w-s$ we get
\begin{equation*}
\sum_{s=2}^{w-2} \Big(8s-4w\Big)Z_\calA^\eff(\ol{s})Z_\calA^\eff(\ol{r})=0.
\end{equation*}
We have now completed the proof of the proposition.
\end{proof}

\subsection{Effective Euler sums $Z_\calA^\eff(\ol{a},\ol{b})$}
We consider the last type of effective Euler sums $Z_\calA^\eff(\ol{a},\ol{b})$ in this section.

\begin{thm}\label{thm:KZeffDBES(-1,-1)}
Assume $b\ge 2$ and $w=a+b\ge 4$ is even. Then, modulo $\zeta(2)$ we have
\begin{align}\label{equ:KZrel1SMES(-1,-1)version}
&\,  \zeta_\sha^\Sy(\ol{a},\ol{1},b-1)
\equiv   -(b-1)\zeta(\ol{a})\zeta(b)-\zeta(\ol{a})\zeta(\ol{b})
 +(-1)^{a}\sum_{s=1}^{w-1}\binom{s}{a}\zeta_*(\ol{s},\ol{w-s}),\\
&\sum_{s=1}^{w-1} \zeta(\ol{s},\ol{w-s})\equiv \zeta(\ol{1},w-1).\label{equ:KZrel2SMES(-1,-1)version}
\end{align}
\end{thm}

\begin{proof}
Taking $\eps_1=\eps_2=-1$ and $\eps_3=1$ in  Theorem~\ref{thm:SMZV(a,1,b-1;e1,e2,e3)}, we get
\begin{align*}
&\,  \zeta_\sha^\Sy(\ol{a},\ol{1},b-1)
\equiv   -(b-1)\zeta(\ol{a})\zeta(b)-\zeta(\ol{a})\zeta(\ol{b})+ (-1)^a\zeta(\ol{a+1})\zeta_*(b-1) \\
&\,-(-1)^{a}\sum_{s=1}^{w-1}\left[ \binom{s-1}{b-2} \zeta_*(s,\ol{w-s}) -\binom{s-1}{a-1}\zeta_*(\ol{s},\ol{w-s})\right] .
\end{align*}
Note that we multiplied by $(-1)^a$ in front of $\zeta(\ol{a+1})\zeta_*(b-1)$ since this term could be nonzero only when $a$ is even. By shuffle product formula (see \cite[(3)]{GanglKaZa2006} for the corresponding result for double zetas)
\begin{align}\label{equ:shuffle_b_bara}
\zeta_*(b-1)\zeta(\ol{a+1})=&\, \sum_{s=1}^{w-1} \bigg[\binom{s-1}{b-2}\zeta_*(s,\ol{w-s}) +\binom{s-1}{a}\zeta_*(\ol{s},\ol{w-s})\bigg]
\end{align}
which follows from
\begin{align*}
\ta^{b-2}\tb \sha \ta^a\tc =&\, \sum_{s=1}^{w-1} \bigg[\binom{s-1}{b-2}\ta^{s-1}\tb \ta^{w-s-1}\tc +\binom{s-1}{a} \ta^{s-1}\tc\ta^{w-s-1}\tb\bigg]  .
\end{align*}
These yields \eqref{equ:KZrel1SMES(-1,-1)version} quickly. Finally, \eqref{equ:KZrel2SMES(-1,-1)version} is the special case of \eqref{equ:shuffle_b_bara} when $a=0$ since
\begin{align*}
\zeta(\ol{1})\zeta(w-1)-\zeta(w-1,\ol{1})\equiv \zeta(\ol{1},w-1) \pmod{\zeta(2)}.
\end{align*}
This concludes the proof of the theorem.
\end{proof}

\begin{cor}\label{cor:KZeffDBES(-1,-1)}
Suppose $a,b$ are positive integers with even $w=a+b\ge 4$. If $1\le a\le w-2$ then we have
\begin{align} \label{equ:KZrel3SMES(-1,-1)}
\zeta(\ol{a},\ol{b})\equiv &\, -(-1)^a \binom{w-1}{a}\zeta(\ol{1},w-1)\\
&\,+ (-1)^{a}\sum_{s=1}^{w-2} \left[ \binom{s}{a}-\binom{w-1}{a}\right]
\left\{\aligned
&\zeta_\ast^\Sy(\ol{s},\ol{1},w-s-1)+\zeta(\ol{s})\zeta(\ol{w-s}) \\
&+(w-s-1)\zeta(\ol{s})\zeta(w-s)\endaligned \right\}, \notag\\     
Z_\calA^\eff(\ol{a},\ol{b})=&\,-(-1)^a \binom{w-1}{a} Z_\calA^\eff(\ol{1},w-1)  \notag\\
&\,+ (-1)^{a}\sum_{s=1}^{w-2} \left[ \binom{s}{a}-\binom{w-1}{a}\right]
\left\{\aligned
&\zeta_\calA(\ol{s},\ol{1},w-s-1)+Z_\calA^\eff(\ol{s})Z_\calA^\eff(\ol{w-s}) \\
&+(w-s-1)Z_\calA^\eff(\ol{s})Z_\calA^\eff(w-s)\endaligned \right\}, \notag
\end{align}
where $Z_\calA^\eff(\ol{1},w-1)$ is given by \eqref{equ::KZrel3SMES(1,-1)eg}. Further,
\begin{align}
Z_\calA^\eff(\ol{w-1},\ol{1})=Z_\calA^\eff(\ol{1},w-1)-\sum_{s=1}^{w-2} Z_\calA^\eff(\ol{s},\ol{w-s}) .  \label{equ:ConjSumFormulaFDESbarbar}
\end{align}
\end{cor}

\begin{rem}
The sum formula for double Euler sums corresponding to \eqref{equ:ConjSumFormulaFDESbarbar} is given by  \cite[Theorem 4.2]{BCJXXZhao2020c}.
\end{rem}

\begin{proof}
The proof is again similar to that of Cor.~\ref{cor:KZeffDBZs}. We only need to prove \eqref{equ:KZrel3SMES(-1,-1)}  which follows from Lemma~\ref{lem:linAlg} by taking $A=1$, $B=-\zeta(\ol{1},w-1)$,
$x_s=\zeta(\ol{s},\ol{w-s})$  and
\begin{align*}
\psi_a=\zeta_\ast^\Sy(\ol{a},\ol{1},b-1) +\zeta(\ol{a})\zeta(\ol{b}) +(b-1)\zeta(\ol{a})\zeta(b)
  \quad \text{for } 1\le a\le w-2,
\end{align*}
where $b=w-a$.
\end{proof}

The odd weight version of the following proposition is given by Prop.~\ref{prop:SumFormula}.

\begin{prop}\label{prop:SumFormulaBarab}   \emph{(Sum formulas})
For all positive even integer $w\ge 3$, setting $v=w-1$ then we have
\begin{align}
  \sum_{\substack{a\ge 2,b\ge 1\\ a+b=w}}  Z_\calA^\eff(\ol{a},\ol{b})=&\,  Z_\calA^\eff(\ol{1},v)- Z_\calA^\eff(\ol{1},\ol{v}), \label{EvenWtSumBaraBarb}
\end{align}
\end{prop}

\begin{proof}
The follows from \eqref{equ:ConjSumFormulaFDESbarbar} immediately.
\end{proof}

It is natural to expect the following relations to hold. To prove these, however, one first needs the triple Euler sum analog of \cite[Cor.~5.9]{BachmannRisan2026}, which we have verified numerically only. Further, \eqref{EvenWtSumBarab} follows easily from the $a=1$ case of the stuffle relation \eqref{equ:ConjStuffFDESbarbar}.

\begin{conj}
Let $w\ge 4$ be an even integer. Then, for all $1\le a\le w-1$, we have
\begin{align}
  \sum_{\substack{a\ge 2,b\ge 1\\ a+b=w}}  Z_\calA^\eff(\ol{a},b)=&\,  Z_\calA^\eff(\ol{v},\ol{1})+ Z_\calA^\eff(\ol{1},\ol{v})-Z_\calA^\eff(\ol{v},1)-Z_\calA^\eff(\ol{1},v), \label{EvenWtSumBarab}\\
  \sum_{\substack{a\ge 2,b\ge 1\\ a+b=w}}  Z_\calA^\eff(a,\ol{b})=&\, Z_\calA^\eff(\ol{v},1)-Z_\calA^\eff(\ol{v},\ol{1}),\\
Z_\calA^\eff(\ol{a})Z_\calA^\eff(\ol{w-a})=&\, Z_\calA^\eff(\ol{a},\ol{w-a})+ Z_\calA^\eff(\ol{w-a},\ol{a}),\label{equ:ConjStuffFDESbarbar}\\
\sum_{\substack{a\ge 2,b\ge 1\\ a+b=w}}  2^a \Big(Z_\calA^\eff(\ol{a},\ol{b})+ Z_\calA^\eff(a,\ol{b})\Big)=&\, 0. \label{equ:EvenWtSum(a,barb)AND(a,barb)}
\end{align}
\end{conj}

\section{Concluding Remarks}
In the main body of this paper, we have concentrated on constructing effective single and double effective Euler sums that correspond to the classical one under Conjecture~\ref{conj:KanekoZagierAltVersion}. At higher depths, the number of possible sign patterns and compositions grows rapidly, and the space of relations becomes more complicated. Moreover, the two ways of regularization will enter the picture in an essential way, leading to two models of ``effective'' Euler sums. Nevertheless, the low-depth formulas indicate that the finite/classical correspondence contains considerably more explicit information than an abstract isomorphism alone would suggest. In particular, one may hope to construct a compatible family of effective representatives that respects products, dualities, and the various natural filtrations on the two sides.

We emphasize that the purpose of the present work is therefore not merely to verify isolated instances of a conjectural correspondence. Rather, we seek to make the correspondence itself explicit. The resulting formulas provide concrete objects on the finite side attached to classical Euler sums (and their regularized values for divergent ones) and thereby turn an abstract conjectural isomorphism into a calculable procedure. Therefore, it would be very interesting to construct effect Euler sums explicitly at general depth.

\bigskip

\appendix

\renewcommand{\thesection}{Appendix.}
\section{Heuristic Approach to Effective Double Euler Sums}

\begin{center}
by Ryan Shi and Jianqiang Zhao
\end{center}

\renewcommand{\thesection}{A}

\bigskip

We now pursue the question of constructing effective Euler sums from a complementary direction in the case of depth two and even weight, by giving another family of plausible effective Euler sums, motivated by the heuristic argument of Kaneko and Zagier for double zeta values in \cite{KanekoZa2026}. The argument suggests that one should regard the finite/classical correspondence not only as a map between already-established vector spaces, but as a mechanism for transporting the classical series definition into the finite setting. Extending this heuristic to Euler sums leads to candidates that are different in appearance from the formulas obtained in the main body of the paper, while exhibiting the same expected behavior. These candidates provide additional evidence for the conjectural correspondence and may also point toward a more conceptual construction of effective representatives in arbitrary depth.

\subsection{Effective double zetas}
We first recall Kaneko and Zagier's definition in the double zeta case.

\begin{defn} \label{defn:KZdefn} (\cite[\S 5]{KanekoZa2026})
For all $l\ge 0$ and $i\ge 1$, set
\begin{equation*}
    \gb(l,i):=-\frac{1}{l+1} \binom{l+1}{i} B_{l+1-i}.
\end{equation*}
For $a\ge 2$ and $b\ge 1$, let $v=a+b-1$. If $a+b$ is \emph{even} then we can define
\begin{equation*}
Z_\calA(a,b):=\big(Z_p(a,b) \pmod{p}\big)_{p\in \calP,p>a+b}
\end{equation*}
where for all primes $p$
\begin{align*}
Z_p(a,b):=&\, \sum_{i=1}^{a-2} (-1)^b \binom{b+i}{b}Z_p^\eff(b+i)Z_p^\eff(a-i)+\sum_{j=b-1}^{p-1-a} \gb(j,b)Z_p^\eff(a+j)-Z_p^\eff(a+b)\\
&\, +(-1)^b \binom{v}{b}\Big[(H_{a-1}-H_{v}+\gamma_\calA )Z_\calA^\eff(v) - Z'_\calA(v)\Big].
\end{align*}
Here $\gamma_\calA=(\gamma_p )_{p\in\calP}$ is the finite analog of Euler's gamma constant where
\begin{equation*}
    \gamma_p \equiv \frac{(p-1)!+1}{p} \pmod{p},
\end{equation*}
and
\begin{equation}\label{equ:Zder}
    Z'_\calA(v):=\left(\frac1{p}\left( \frac{B_{2p-1-v}}{2p-1-v}-\frac{B_{p-v}}{p-v} \right)\mod p\right)_{p\in\calP}\in \calA.
\end{equation}
\end{defn}

We have noticed that Kina \cite{Kina2026} recently obtained independently some the same results in the main body of this paper in the MZV setting via a completely different approach. Using one of his results we now prove that Kaneko and Zagier's heuristically defined effective double zeta values coincide with ours.

\begin{thm}\label{thm:KZversion=mine}
For all $a,b\in\N$ with $a+b$ \emph{even}, we have
\begin{equation*}
Z_\calA(a,b)=Z_\calA^\eff(a,b).
\end{equation*}
\end{thm}
\begin{proof}
In \cite{Kina2026}, Kina proved that
\begin{equation*}
Z_\calA(a,b)+Z_\calA^\eff(b,a)=Z_\calA^\eff(a)Z_\calA^\eff(b).
\end{equation*}
The theorem now follows from the stuffle relation Prop.~\ref{prop:doubleZetaProperties}(i).
\end{proof}

\subsection{Effective double Euler sums of types $(\ol{a},b)$ and $(a,\ol{b})$}
We now adapt the above ideas to other double Euler sums.

\begin{defn}\label{defn:zBar12}
Suppose $a,b\in\N$ with $a+b$ \emph{even}. Define $Z_\calA(\ol{a},b)=\big(Z_p(\ol{a},b) \pmod{p}\big)_{p\in\calP}$
where for all primes $p$
\begin{equation*}
Z_p(\ol{a},b):=\sum_{i=1}^{a-1} (-1)^b \binom{b+i}{b} Z_p^\eff(b+i)Z_p^\eff(\ol{a-i})  +\sum_{j=b-1}^{p-1-a} \gb(j,b)Z_p^\eff(\ol{a+j})-Z_p^\eff(\ol{a+b}).
\end{equation*}
Further, define
\begin{equation*}
Z_p(a,\ol{b}):= Z_p^\eff(a)Z_p^\eff(\ol{b})-Z_p(\ol{b},a).
\end{equation*}
\end{defn}

\begin{conj} \label{conj:zBar12}
Suppose $a,b\in\N$ with $a+b$ \emph{even}. Then $Z_\calA(\ol{a},b)=Z_\calA^\eff(\ol{a},b)$ and $Z_\calA(a,\ol{b})=Z_\calA^\eff(a,\ol{b})$.
\end{conj}

We provide the following heuristic argument for the correctness of the definition for $Z_p(\ol{a},b)$. Then the definition for $Z_p(a,\ol{b})$ follows easily from the stuffle relation since $Z_p^\eff(\ol{a+b})=0$.
Recall the Seki-Bernoulli formula can be used to express the sum of powers
\begin{equation*}
S_l(n):=\sum_{j=1}^{n} j^l=\sum_{i=1}^{l+1} (-1)^{l-i} \gb(l,i) n^i
\end{equation*}
for all $l,n\in\N$. By \cite[(60)]{KanekoZa2026} we see that ``modulo prime $p$''
\begin{align*}
  \zeta^\star (\ol{a},b) =\phantom{'} &\, \sum_{n=1}^\infty \sum_{j=1}^{n} \frac{(-1)^n}{n^a j^b}\\
 \fkeqv &\, \sum_{n=1}^\infty  (-1)^n n^{-a} \sum_{j=1}^{n}  j^{p-1-b}   \\
  \equiv\phantom{'}&\, \sum_{n=1}^\infty (-1)^n \sum_{i=1}^{p-b} \gb(p-1-i,b) n^{i-a}  \\
  \equiv\phantom{'}&\, \sum_{n=1}^\infty (-1)^n\left( \sum_{i=1}^{a-1} \gb(p-1-i,b) n^{i-a}
  + \sum_{j=b-1}^{p-1-a} \gb(j,b) n^{p-1-j-a}  \right)  \\
  \fkeq &\,  \sum_{i=1}^{a-1} \gb(p-1-i,b) \zeta (\ol{a-i})   + \sum_{j=b-1}^{p-1-a} \gb(j,b) \zeta\big(\ol{1-(p-a-j)}\big)
\end{align*}
where, for $n>0$, we understand $\zeta(\ol{-n})$ as the value of the analytically continued $\zeta(\ol{s})=(2^{1-s}-1)\zeta(s)$ at $s=-n$.
As $\zeta(1-n)=-B_n/n$ for all $n>1$ and
\begin{equation*}
\gb(p-1-i,b)=-\frac{1}{p-i} \binom{p-i}{b} B_{p-b-i}\equiv (-1)^b \binom{b+i}{b} \frac{B_{p-b-i}}{b+i} \pmod{p},
\end{equation*}
we are led to Definition~\ref{defn:zBar12} immediately.

\end{document}